\documentclass[11pt]{amsart}
\usepackage{amsmath,amssymb,hyperref,geometry,graphicx,tikz,ytableau,fp,mathdots,enumitem,xcolor,mathtools}
\usepackage{cmbright}

\newcommand{\R}{\mathbb{R}}
\newcommand{\Z}{\mathbb{Z}}

\newcommand{\rk}{\mathrm{rk}}

\renewcommand{\phi}{\varphi}
\newcommand{\PLP}{\mathrm{PLaur}}
\newcommand{\PLM}{\mathrm{PLaur}}
\newcommand{\M}{\mathsf{M}} 
\newcommand\Def[1]{\emph{#1}}

\newcommand{\PL}{\mathrm{PLin}}
\newcommand{\PE}{\mathrm{PExp}}
\newcommand{\uPE}{\underline{\mathrm{PExp}}}
\newcommand{\uPL}{\underline{\mathrm{PLin}}}

\renewcommand{\L}{\mathrm{Lin}}
\renewcommand{\emptyset}{\varnothing}
\newcommand{\Q}{\mathbb{Q}}

\newtheorem{dummy}{}[section]
\newtheorem{theorem}[dummy]{Theorem} 
\newtheorem{corollary}[dummy]{Corollary}
\newtheorem{proposition}[dummy]{Proposition}

\newtheorem{lemma}[dummy]{Lemma}
\theoremstyle{definition}

\newtheorem{remark}[dummy]{Remark}

\begin{document}

\title{A matroidal twist on a formula of Brion}

\author[M.~Beck]{Matthias Beck}
\address{Department of Mathematics, San Francisco State University, United States of America}
\email{mattbeck@sfsu.edu}

\author[C.~Klivans]{Caroline Klivans}
\address{Department of Mathematics, Brown University, United States of America}
\email{caroline\_klivans@brown.edu}

\author[D.~Ross]{Dustin Ross}
\address{Department of Mathematics, San Francisco State University, United States of America}
\email{rossd@sfsu.edu}


\begin{abstract}
Brion's Formula realizes the Laurent polynomial of lattice points in a lattice polytope $P$ as the sum of rational functions associated to the vertices of $P$. In this paper, we consider the special case where $P$ is a generalized permutohedron. We study a modification of the rational functions associated to the vertices of $P$ depending on a given matroid $\M$.  Upon summing these rational functions, we describe how the resulting Laurent polynomial $Q_\M(P)$ behaves in certain ways like the lattice points of $P$, exhibiting natural recursive and reciprocity behaviors. Furthermore, upon evaluating $Q_\M(P)$ at $1$, we recover the matroid Euler characteristics of Larson, Li, Payne, and Proudfoot, so the combinatorial approach in this paper gives new insight into studying these quantities. 
\end{abstract}

\maketitle


\vspace{-18bp}

\section{Introduction}

Given a lattice polytope $P\subset\R^n$, the integer lattice points of $P$ are encoded in the Laurent polynomial
\[
q(P) \, \coloneqq  \sum_{m\in P\cap\Z^n}x^m\in\Z[x_1^{\pm 1},\dots, x_n^{\pm 1}] \, ,
\]
where $x^m\coloneqq x_1^{m_1}\cdots x_n^{m_n}$. Let $V(P)$ denote the vertices of $P$, and for each vertex $v\in V(P)$,
consider the \Def{affine vertex cone} of $P$ at~$v$:
\[
C_{P,v} \, \coloneqq  \, v+\R_{\geq 0}(P-v) \, \subset \, \R^n .
\]
The formal generating series $q(C_{P,v})$ of lattice points in $C_{P,v}$ can be viewed as the expansion of a rational
function $Q(C_{P,v})\in\Z(x_1,\dots,x_n)$, and \emph{Brion's Formula} \cite{brion} says that, when we sum these rational functions over all vertices, the poles cancel, and the sum remarkably simplifies to the Laurent polynomial of lattice points in $P$:
\begin{equation}\label{eq:BrionClassical}
\sum_{v\in V(P)}Q(C_{P,v}) \, = \, q(P) \, .
\end{equation}
Figure~\ref{fig:Brion} depicts an example of Brion's Formula. See~\cite{MR2480796} for expositions of this and adjacent~results.

\begin{figure}[ht]
\begin{center}
\begin{tikzpicture}[baseline={(0,0)},scale=.35]
\fill[green!20, fill opacity=.5] (0,0) -- (4.5,0) -- (4.5,4.5) -- (0,4.5) -- (0,0);
\draw[thick,->] (0,0) -- (5,0);
\draw[thick,->] (0,0) -- (0,5);
\node at (0,0) [draw, shape=circle, fill=black, minimum size=2.5pt, inner sep=0pt] {};
\node at (1,0) [draw, shape=circle, fill=black, minimum size=2.5pt, inner sep=0pt] {};
\node at (2,0) [draw, shape=circle, fill=black, minimum size=2.5pt, inner sep=0pt] {};
\node at (3,0) [draw, shape=circle, fill=black, minimum size=2.5pt, inner sep=0pt] {};
\node at (4,0) [draw, shape=circle, fill=black, minimum size=2.5pt, inner sep=0pt] {};
\node at (0,1) [draw, shape=circle, fill=black, minimum size=2.5pt, inner sep=0pt] {};
\node at (1,1) [draw, shape=circle, fill=black, minimum size=2.5pt, inner sep=0pt] {};
\node at (2,1) [draw, shape=circle, fill=black, minimum size=2.5pt, inner sep=0pt] {};
\node at (3,1) [draw, shape=circle, fill=black, minimum size=2.5pt, inner sep=0pt] {};
\node at (4,1) [draw, shape=circle, fill=black, minimum size=2.5pt, inner sep=0pt] {};
\node at (0,2) [draw, shape=circle, fill=black, minimum size=2.5pt, inner sep=0pt] {};
\node at (1,2) [draw, shape=circle, fill=black, minimum size=2.5pt, inner sep=0pt] {};
\node at (2,2) [draw, shape=circle, fill=black, minimum size=2.5pt, inner sep=0pt] {};
\node at (3,2) [draw, shape=circle, fill=black, minimum size=2.5pt, inner sep=0pt] {};
\node at (4,2) [draw, shape=circle, fill=black, minimum size=2.5pt, inner sep=0pt] {};
\node at (0,3) [draw, shape=circle, fill=black, minimum size=2.5pt, inner sep=0pt] {};
\node at (1,3) [draw, shape=circle, fill=black, minimum size=2.5pt, inner sep=0pt] {};
\node at (2,3) [draw, shape=circle, fill=black, minimum size=2.5pt, inner sep=0pt] {};
\node at (3,3) [draw, shape=circle, fill=black, minimum size=2.5pt, inner sep=0pt] {};
\node at (4,3) [draw, shape=circle, fill=black, minimum size=2.5pt, inner sep=0pt] {};
\node at (0,4) [draw, shape=circle, fill=black, minimum size=2.5pt, inner sep=0pt] {};
\node at (1,4) [draw, shape=circle, fill=black, minimum size=2.5pt, inner sep=0pt] {};
\node at (2,4) [draw, shape=circle, fill=black, minimum size=2.5pt, inner sep=0pt] {};
\node at (3,4) [draw, shape=circle, fill=black, minimum size=2.5pt, inner sep=0pt] {};
\node at (4,4) [draw, shape=circle, fill=black, minimum size=2.5pt, inner sep=0pt] {};
\node [below left]  at (0,0) {\tiny $(0,0)$};
\node[font=\small] at (2.2,-3.5) { $\frac{1}{(1-x)(1-y)}$ };
\end{tikzpicture}
\begin{tikzpicture}[baseline={(0,0)},scale=.35]
\node[font=\small] at (1,-3.5) {$\;\;+\;\;\;$};
\end{tikzpicture}
\begin{tikzpicture}[baseline={(0,0)},scale=.35]
\fill[green!20, fill opacity=.5] (-1.5,0) -- (3,0) -- (-1.5,4.5) -- (-1.5,0);
\draw[thick,->] (3,0) -- (-2,0);
\draw[thick,->] (3,0) -- (-2,5);
\node at (-1,0) [draw, shape=circle, fill=black, minimum size=2.5pt, inner sep=0pt] {};
\node at (0,0) [draw, shape=circle, fill=black, minimum size=2.5pt, inner sep=0pt] {};
\node at (1,0) [draw, shape=circle, fill=black, minimum size=2.5pt, inner sep=0pt] {};
\node at (2,0) [draw, shape=circle, fill=black, minimum size=2.5pt, inner sep=0pt] {};
\node at (3,0) [draw, shape=circle, fill=black, minimum size=2.5pt, inner sep=0pt] {};
\node at (-1,1) [draw, shape=circle, fill=black, minimum size=2.5pt, inner sep=0pt] {};
\node at (0,1) [draw, shape=circle, fill=black, minimum size=2.5pt, inner sep=0pt] {};
\node at (1,1) [draw, shape=circle, fill=black, minimum size=2.5pt, inner sep=0pt] {};
\node at (2,1) [draw, shape=circle, fill=black, minimum size=2.5pt, inner sep=0pt] {};
\node at (-1,2) [draw, shape=circle, fill=black, minimum size=2.5pt, inner sep=0pt] {};
\node at (0,2) [draw, shape=circle, fill=black, minimum size=2.5pt, inner sep=0pt] {};
\node at (1,2) [draw, shape=circle, fill=black, minimum size=2.5pt, inner sep=0pt] {};
\node at (-1,3) [draw, shape=circle, fill=black, minimum size=2.5pt, inner sep=0pt] {};
\node at (0,3) [draw, shape=circle, fill=black, minimum size=2.5pt, inner sep=0pt] {};
\node at (-1,4) [draw, shape=circle, fill=black, minimum size=2.5pt, inner sep=0pt] {};
\node [below right] at (3,0) {\tiny $(3,0)$};
\node[font=\small] at (.8,-3.5) { $\frac{x^3}{(1-x^{-1})(1-x^{-1}y)}$ };
\end{tikzpicture}
\begin{tikzpicture}[baseline={(0,0)},scale=.35]
\node[font=\small] at (1.5,-3.5) {$+\;\;\;\;$};
\end{tikzpicture}
\begin{tikzpicture}[baseline={(0,0)},scale=.35]
\fill[green!20, fill opacity=.5] (0,-1.5) -- (4.5,-1.5) -- (0,3) -- (0,-1.5);
\draw[thick,->] (0,3) -- (0,-2);
\draw[thick,->] (0,3) -- (5,-2);
\node at (0,-1) [draw, shape=circle, fill=black, minimum size=2.5pt, inner sep=0pt] {};
\node at (1,-1) [draw, shape=circle, fill=black, minimum size=2.5pt, inner sep=0pt] {};
\node at (2,-1) [draw, shape=circle, fill=black, minimum size=2.5pt, inner sep=0pt] {};
\node at (3,-1) [draw, shape=circle, fill=black, minimum size=2.5pt, inner sep=0pt] {};
\node at (4,-1) [draw, shape=circle, fill=black, minimum size=2.5pt, inner sep=0pt] {};
\node at (0,0) [draw, shape=circle, fill=black, minimum size=2.5pt, inner sep=0pt] {};
\node at (1,0) [draw, shape=circle, fill=black, minimum size=2.5pt, inner sep=0pt] {};
\node at (2,0) [draw, shape=circle, fill=black, minimum size=2.5pt, inner sep=0pt] {};
\node at (3,0) [draw, shape=circle, fill=black, minimum size=2.5pt, inner sep=0pt] {};
\node at (0,1) [draw, shape=circle, fill=black, minimum size=2.5pt, inner sep=0pt] {};
\node at (1,1) [draw, shape=circle, fill=black, minimum size=2.5pt, inner sep=0pt] {};
\node at (2,1) [draw, shape=circle, fill=black, minimum size=2.5pt, inner sep=0pt] {};
\node at (0,2) [draw, shape=circle, fill=black, minimum size=2.5pt, inner sep=0pt] {};
\node at (1,2) [draw, shape=circle, fill=black, minimum size=2.5pt, inner sep=0pt] {};
\node at (0,3) [draw, shape=circle, fill=black, minimum size=2.5pt, inner sep=0pt] {};
\node [above] at (0,3) {\tiny$(0,3)$};
\node[font=\small] at (2.2,-3.5) { $\frac{y^3}{(1-y^{-1})(1-xy^{-1})}$ };
\end{tikzpicture}
\begin{tikzpicture}[baseline={(0,0)},scale=.35]
\node[font=\small] at (1.5,-3.5) {$\;\;\;\;\;\;=$};
\end{tikzpicture}
\begin{tikzpicture}[baseline={(0,0)},scale=.35]
\draw[thick,fill=green!20, fill opacity=.5] (0,0) -- (3,0) -- (0,3) -- (0,0);
\node at (0,0) [draw, shape=circle, fill=black, minimum size=2.5pt, inner sep=0pt] {};
\node [below left]  at (0,0) {\tiny $(0,0)$};
\node at (1,0) [draw, shape=circle, fill=black, minimum size=2.5pt, inner sep=0pt] {};
\node at (2,0) [draw, shape=circle, fill=black, minimum size=2.5pt, inner sep=0pt] {};
\node at (3,0) [draw, shape=circle, fill=black, minimum size=2.5pt, inner sep=0pt] {};
\node [below right] at (3,0) {\tiny $(3,0)$};
\node at (0,1) [draw, shape=circle, fill=black, minimum size=2.5pt, inner sep=0pt] {};
\node at (1,1) [draw, shape=circle, fill=black, minimum size=2.5pt, inner sep=0pt] {};
\node at (2,1) [draw, shape=circle, fill=black, minimum size=2.5pt, inner sep=0pt] {};
\node at (0,2) [draw, shape=circle, fill=black, minimum size=2.5pt, inner sep=0pt] {};
\node at (1,2) [draw, shape=circle, fill=black, minimum size=2.5pt, inner sep=0pt] {};
\node at (0,3) [draw, shape=circle, fill=black, minimum size=2.5pt, inner sep=0pt] {};
\node [above] at (0,3) {\tiny$(0,3)$};
\node[align=left,font=\tiny] at (2,-3.5) { $\phantom{+}y^3$\\ $+y^2+xy^2$\\  $+y+xy+x^2y$\\  $+1+x+x^2+x^3$ };
\end{tikzpicture}
\end{center}
\caption{Brion's Formula.}\label{fig:Brion}
\end{figure}
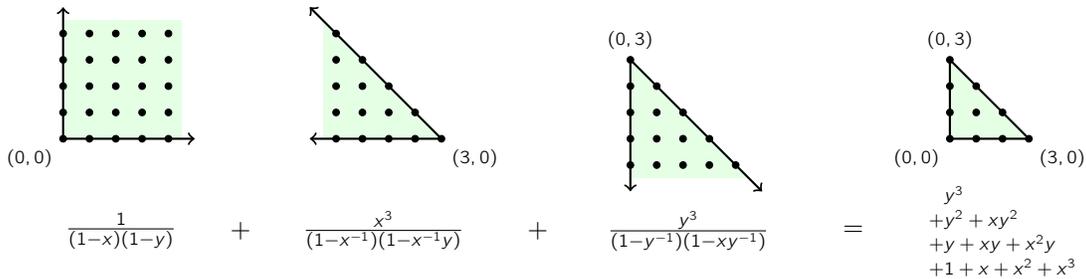

Our goal is to study matroidal generalizations of Brion's Formula in the special setting of generalized
permutohedra. More specifically, for each generalized permutohedron $P\subset\R^n$ and each matroid $\M$ on $[n]
\coloneqq  \{ 1, 2, \dots, n \}$, we study a modification of the left-hand side of \eqref{eq:BrionClassical} that
incorporates the matroid $\M$ in a nontrivial way. The poles still cancel when summed, yielding a
Laurent polynomial $Q_\M(P)$ with interesting properties that naturally generalize fundamental properties of
$q(P)$. The rest of this section aims to briefly explain the combinatorial modification and the claimed implications. 
(See Section~\ref{sec:prerequisites} for more background information on generalized permutohedra and matroids).

A \Def{(lattice) generalized permutohedron} is a lattice polytope $P\subset\R^n$ whose normal fan coarsens the real braid fan $\Sigma_n$.
The prototype of a generalized permutohedron is the \Def{standard permutohedron}, which is the convex hull of
all vectors obtained by permuting the coordinates of $(1,2,\dots,n)\in\R^n$. 
We say that a generalized
permutohedron $P\subset\R^n$ is \Def{nondegenerate} if its normal fan is equal to the braid fan, as is the case for the standard permutohedron. The maximal cones of the braid fan are indexed by permutations in $S_n$, so for any nondegenerate generalized permutohedron $P\subset\R^n$, there is a bijection
$v_P:S_n\rightarrow V(P)$. Moreover, we can write the associated affine vertex cone $C_{P,v_P(\sigma)}$ explicitly in terms of the permutation~$\sigma$:
\[
C_{P,v_P(\sigma)} \, = \, v_P(\sigma)+C_\sigma\;\;\;\text{ where }\;\;\;C_\sigma \, = \, \R_{\geq 0}(e_{\sigma(2)}-e_{\sigma(1)},\dots,e_{\sigma(n)}-e_{\sigma(n-1)}) \, .
\]
It then follows that Brion's Formula for nondegenerate generalized permutohedra can be written explicitly as a sum over permutations:
\begin{equation}\label{eq:Brion}
\sum_{\sigma\in S_n}x^{v_P(\sigma)}Q(C_\sigma) \, = \, q(P)\;\;\;\text{ where }\;\;\;Q(C_\sigma) \, = \,
\frac{1}{\prod_{i=1}^{n-1}\big(1-\frac{x_{\sigma(i+1)}}{x_{\sigma(i)}}\big)} \, .
\end{equation}
More generally, if $P\subset\R^n$ is a (possibly degenerate) generalized permutohedron, there is still a
surjection $v_P:S_n\rightarrow V(P)$, and \eqref{eq:Brion} continues to hold in this more general
setting~\cite{ishida}. Even more generally, if $f=\{f_\sigma\}$ is any piecewise Laurent polynomial on the braid fan, then \eqref{eq:Brion} continues to hold upon replacing $x^{v_P(\sigma)}$ with $f_\sigma$ and by replacing $P$ with the virtual polytope associated to $f$ \cite{VP1,VP2}.

A \Def{matroid} $\M$ on the ground set $[n]$ is a nonempty collection of subsets of $[n]$, called \emph{bases}, that satisfy the basis-exchange property. The \Def{base polytope} $P_\M\subset\R^n$ of a matroid $\M$ is the convex hull of the vectors $e_B\coloneqq \sum_{i\in B} e_i$ where
$B\in \M$ is a basis. The basis-exchange property implies that $P_\M$ is a generalized permutohedron with
vertices in bijection with bases of $\M$; in particular, there is a surjection $B_\M:S_n\rightarrow\M$; the basis $B_\M(\sigma)$ is often called the \emph{greedy basis} of $\M$ with respect to $\sigma$.

Now let $f\in \PLP(\Sigma_n)$ be a piecewise Laurent polynomial on the braid fan. In this paper, we study the following generalizations of the rational function appearing in the left-hand side of \eqref{eq:Brion}:
\[
Q_\M(f) \, \coloneqq  \, \sum_{\sigma\in S_n}f_\sigma Q(C_\sigma)\prod_{i\notin B_\M(\sigma)}(1-x_i) \, \in \,
\Z(x_1,\dots,x_n) \, .
\]
In the special case where $f=\{x^{v_P(\sigma)}\}$ is the collection of vertex monomials of a generalized permutohedron
$P\subset\R^n$, we use the notation $Q_\M(P)\coloneqq Q_\M(f)$. Our primary motivation for studying $Q_\M(f)$ is that
it naturally arises when one computes matroid Euler characteristics, as defined by Larson, Li, Payne, and Proudfoot
\cite{larsonetal}, via localization techniques on the permutohedral variety, following ideas developed by Berget, Eur, Spink, and Tseng \cite{BEST}.
While many of the results in this paper can be proved directly using results from \cite{BEST} and \cite{larsonetal}, we approach the study of $Q_\M(f)$ from the purely combinatorial perspective. In particular, our aim is to start with the expression for $Q_M(f)$ given above and to give elementary justifications of its most fundamental properties. This first-principles approach to studying $Q_M(f)$ serves to reveal novel combinatorial aspects of equivariant Euler characteristics of matroids, while also providing a more concrete entryway into their study.

While $Q_\M(f)$ is defined as a sum of rational functions, the most fundamental property of $Q_\M(f)$ is that the poles appearing in the summands cancel in the sum, yielding a Laurent polynomial.

\begin{theorem}\label{thm:Polynomial}
For any piecewise Laurent polynomial $f\in \PLP(\Sigma_n)$ and any matroid $\M$ on $[n]$, 
\[
Q_\M(f) \, \in \, \Z[x_1^{\pm 1},\dots,x_n^{\pm 1}] \, .
\]
\end{theorem}

Upon observing that the products in each summand of $Q_\M(f)$ form a piecewise Laurent polynomial on the braid fan (this follows from the basis-exchange axiom), we note that Theorem~\ref{thm:Polynomial} follows from the generalization of Brion's Formula to virtual polytopes \cite[Proposition~4.2]{VP1} (or equivalently, from the $K$-theoretic localization formula on the pemutohedral variety \cite[Theorem~10.2(a)]{BEST}). We give an alternative proof of Theorem~\ref{thm:Polynomial} in Section~\ref{sec:polynomiality}. Our first-principles argument has several notable features: (i)~it gives an explicit recursive construction of the Laurent polynomial $Q_\M(f)$ in terms of $f$ and $\M$, (ii)~it leads to a novel proof of the formulation in \eqref{eq:Brion} of Brion's Formula for generalized permutohedra, and (iii) it gives a simple formula for $Q_\M(1)$.

Equipped with Theorem~\ref{thm:Polynomial}, we now ask the question: For a general
matroid $\M$, in what ways does $Q_\M(P)$ behave like the lattice points in $P$? We present
two answers to this question. 

\subsection{Recursion}

If $P\subset\R^n$ is a nondegenerate generalized permutohedron, then the facets of $P$ correspond to nonempty proper subsets $\emptyset\subset T\subset [n]$. Upon sliding such a facet $F_T$ inward by one lattice step along its normal direction, we obtain a new generalized permutohedron $P_T$, and there is a natural relation between the lattice points of $P$, the lattice points of $P_T$, and the lattice points of $F_T$:
\[
q(P)=q(P_T)+q(F_T)
\] 
This relation is depicted in Figure~\ref{fig:slidingfacet} where $T=\{2,3\}\subset[3]$.

\begin{figure}[ht]
\begin{tikzpicture}[scale=1.2]
\draw[thick,fill=green!20, fill opacity=.5] (1,0) -- (1,1) -- (1-1.732/2,3/2) -- (1-1.732,1) -- (1-1.732,0) -- (1-1.732/2,-1/2) -- (1,0);

\node at (1,0) [draw, shape=circle, fill=black, minimum size=3pt, inner sep=0pt] {};
\node [below right]  at (1,0) {\tiny $(4,2,6)$};
\node at (1,1/2) [draw, shape=circle, fill=black, minimum size=3pt, inner sep=0pt] {};
\node at (1,1) [draw, shape=circle, fill=black, minimum size=3pt, inner sep=0pt] {};
\node [above right] at (1,1) {\tiny $(2,4,6)$};
\node at (1-1.732/4,5/4) [draw, shape=circle, fill=black, minimum size=3pt, inner sep=0pt] {};
\node at (1-1.732/2,3/2) [draw, shape=circle, fill=black, minimum size=3pt, inner sep=0pt] {};
\node [above] at (1-1.732/2,3/2) {\tiny $(2,6,4)$};
\node at (1-3*1.732/4,5/4) [draw, shape=circle, fill=black, minimum size=3pt, inner sep=0pt] {};
\node at (1-1.732,1) [draw, shape=circle, fill=black, minimum size=3pt, inner sep=0pt] {};
\node [above left] at (1-1.732,1) {\tiny $(4,6,2)$};
\node at (1-1.732,1/2) [draw, shape=circle, fill=black, minimum size=3pt, inner sep=0pt] {};
\node at (1-1.732,0) [draw, shape=circle, fill=black, minimum size=3pt, inner sep=0pt] {};
\node [below left] at (1-1.732,0) {\tiny$(6,4,2)$};
\node at (1-3*1.732/4,-1/4) [draw, shape=circle, fill=black, minimum size=3pt, inner sep=0pt] {};
\node at (1-1.732/2,-1/2) [draw, shape=circle, fill=black, minimum size=3pt, inner sep=0pt] {};
\node [below] at (1-1.732/2,-1/2) {\tiny $(6,2,4)$};
\node at (1-1.732/4,-1/4) [draw, shape=circle, fill=black, minimum size=3pt, inner sep=0pt] {};

\node at (1-1.732/2,1/2) [draw, shape=circle, fill=black, minimum size=3pt, inner sep=0pt] {};

\node at (1-1.732/4,1/4) [draw, shape=circle, fill=black, minimum size=3pt, inner sep=0pt] {};
\node at (1-1.732/4,3/4) [draw, shape=circle, fill=black, minimum size=3pt, inner sep=0pt] {};
\node at (1-1.732/2,1) [draw, shape=circle, fill=black, minimum size=3pt, inner sep=0pt] {};
\node at (1-3*1.732/4,3/4) [draw, shape=circle, fill=black, minimum size=3pt, inner sep=0pt] {};
\node at (1-3*1.732/4,1/4) [draw, shape=circle, fill=black, minimum size=3pt, inner sep=0pt] {};
\node at (1-1.732/2,0) [draw, shape=circle, fill=black, minimum size=3pt, inner sep=0pt] {};
\end{tikzpicture}
\hspace{20bp}
\raisebox{46bp}{$=$}
\hspace{20bp}
\begin{tikzpicture}[scale=1.2]
\draw[thick,fill=green!20, fill opacity=.5] (1,0) -- (1,1/2) -- (1-3*1.732/4,5/4) -- (1-1.732,1) -- (1-1.732,0) -- (1-1.732/2,-1/2) -- (1,0);

\node at (1,0) [draw, shape=circle, fill=black, minimum size=3pt, inner sep=0pt] {};
\node [below right]  at (1,0) {\tiny $(4,2,6)$};
\node at (1,1/2) [draw, shape=circle, fill=black, minimum size=3pt, inner sep=0pt] {};
\node [above right] at (1,1/2) {\tiny $(3,3,6)$};
\node at (1-3*1.732/4,5/4) [draw, shape=circle, fill=black, minimum size=3pt, inner sep=0pt] {};
\node [above] at (1-3*1.732/4,5/4) {\tiny $(3,6,3)$};
\node at (1-1.732,1) [draw, shape=circle, fill=black, minimum size=3pt, inner sep=0pt] {};
\node [above left] at (1-1.732,1) {\tiny $(4,6,2)$};
\node at (1-1.732,1/2) [draw, shape=circle, fill=black, minimum size=3pt, inner sep=0pt] {};
\node at (1-1.732,0) [draw, shape=circle, fill=black, minimum size=3pt, inner sep=0pt] {};
\node [below left] at (1-1.732,0) {\tiny$(6,4,2)$};
\node at (1-3*1.732/4,-1/4) [draw, shape=circle, fill=black, minimum size=3pt, inner sep=0pt] {};
\node at (1-1.732/2,-1/2) [draw, shape=circle, fill=black, minimum size=3pt, inner sep=0pt] {};
\node [below] at (1-1.732/2,-1/2) {\tiny $(6,2,4)$};
\node at (1-1.732/4,-1/4) [draw, shape=circle, fill=black, minimum size=3pt, inner sep=0pt] {};

\node at (1-1.732/2,1/2) [draw, shape=circle, fill=black, minimum size=3pt, inner sep=0pt] {};

\node at (1-1.732/4,1/4) [draw, shape=circle, fill=black, minimum size=3pt, inner sep=0pt] {};
\node at (1-1.732/4,3/4) [draw, shape=circle, fill=black, minimum size=3pt, inner sep=0pt] {};
\node at (1-1.732/2,1) [draw, shape=circle, fill=black, minimum size=3pt, inner sep=0pt] {};
\node at (1-3*1.732/4,3/4) [draw, shape=circle, fill=black, minimum size=3pt, inner sep=0pt] {};
\node at (1-3*1.732/4,1/4) [draw, shape=circle, fill=black, minimum size=3pt, inner sep=0pt] {};
\node at (1-1.732/2,0) [draw, shape=circle, fill=black, minimum size=3pt, inner sep=0pt] {};
\end{tikzpicture}
\hspace{20bp}
\raisebox{46bp}{$+$}
\hspace{20bp}
\begin{tikzpicture}[scale=1]
\draw[thick,fill=green!20, fill opacity=.5] (1,1) -- (1-1.732/2,3/2);

\node at (1,1) [draw, shape=circle, fill=black, minimum size=3pt, inner sep=0pt] {};
\node at (1-1.732/4,5/4) [draw, shape=circle, fill=black, minimum size=3pt, inner sep=0pt] {};
\node at (1-1.732/2,3/2) [draw, shape=circle, fill=black, minimum size=3pt, inner sep=0pt] {};
\node [above right] at (1,1) {\tiny $(2,4,6)$};
\node [above] at (1-1.732/2,3/2) {\tiny $(2,6,4)$};
\node [below] at (1-1.732/2,-1/2) {\phantom{5}};
\end{tikzpicture}
\caption{Sliding a facet.}\label{fig:slidingfacet}
\end{figure}
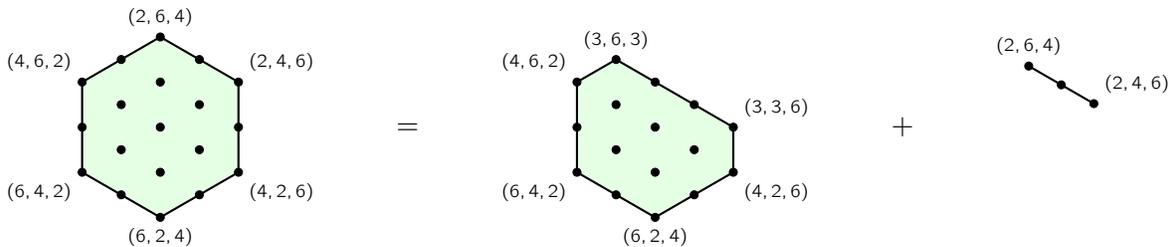

Moreover, the facet $F_T$ can naturally be viewed as the product of two generalized permutohedra, which we denote $P|T\subset \R^T$ and $P/T\subset\R^{T^c}$ \cite{postnikovreinerwilliams}. Thus, we obtain a recursion purely in the context of generalized permutohedra:
\begin{equation}\label{eq:recursionclassical}
q(P)=q(P_T)+q(P|T)q(P/T).
\end{equation}
We refer to \eqref{eq:recursionclassical} as a ``recursion'' because it computes the lattice points of the generalized permutohedron $P$ in terms of the smaller generalized permutohedra $P_T$, $P|T$, and $P/T$.

Abstracting this recursion to any matroid and any piecewise Laurent monomial, we obtain the following generalization of \eqref{eq:recursionclassical} (see Sections~\ref{sec:prerequisites} and \ref{sec:recursion} for precise definitions).

\begin{theorem}\label{thm:recursion}
For any piecewise Laurent monomial $f\in\PLM(\Sigma_n)^*$, for any matroid $\M$ on $[n]$, and for any nonempty proper subset $\emptyset\subset T\subset [n]$, 
\[
Q_\M(f) \, = \, Q_{\M}(f_T)+Q_{\M|T}(f|T) \, Q_{\M/T}(f/T) \, ,
\]
where $\M|T$ is the restriction of $\M$ to $T$ and $\M/T$ is the contraction of $\M$ by $T$.
\end{theorem}

Recent work of Larson, Li, Payne, and Proudfoot \cite{larsonetal} develops the $K$-ring $K(\M)$ of a loopless matroid $\M$ and a matroid Euler characteristic
\[
\chi_\M:K(\M)\rightarrow\Z \, .
\]
For any loopless matroid $\M$, there is a natural surjection $\phi_\M:\PLP(\Sigma_n)\rightarrow K(\M)$, and it follows from the localization perspective developed in \cite{BEST} that $\chi_\M(\phi_\M(f))$ agrees with the evaluation of $Q_\M(f)$ at the value $x_1=\dots=x_n=1$. In other words, the Laurent polynomials $Q_\M(f)$ provide a refined approach to studying Euler characteristics of matroids. In Section~\ref{sec:recursion}, we describe how Theorem~\ref{thm:recursion} and our constructive proof of Theorem~\ref{thm:Polynomial} give an alternative justification of the fact that $\chi_\M(\phi_\M(f))$ agrees with the evaluation of $Q_\M(f)$ at the value $x_1=\dots=x_n=1$, yielding a first-principles justification of the fact that the Bergman fan of a loopless matroid is an \emph{Ehrhart} fan, in the sense of \cite{ehrhartfans}.

\subsection{Reciprocity}

Ehrhart theory studies growth polynomials of lattice points in dilates of a polytope.  A cornerstone of the theory is Ehrhart reciprocity, which asserts that the value of the Ehrhart polynomial of a polytope at $-1$ is, up to a sign, the number of lattice points in the relative interior of the polytope. In the setting of piecewise Laurent polynomials, dilating generalized permutohedra translates to taking powers of the corresponding vertex monomials, so Ehrhart reciprocity concerns the natural map $(-)^\vee:\PLP(\Sigma_n)\rightarrow \PLP(\Sigma_n)$ that sends every monomial to its multiplicative inverse. More generally, given any function on $\PLP(\Sigma_n)$, a \emph{combinatorial reciprocity theorem} for generalized permutohedra aims to compare the values of the function on $f$ and $f^\vee$. Given a matroid $\M$ on $[n]$, we are particularly interested in studying reciprocity for the function $Q_\M:\PLP(\Sigma_n)\rightarrow \Z[x_1^{\pm 1},\dots,x_n^{\pm 1}]$. Define $\omega_\M\in\PLP(\Sigma_n)$ by
\[
(\omega_\M)_\sigma \, \coloneqq  \, \frac{x_{\sigma(n)}}{x_{\sigma(1)}}\prod_{i\notin B_\M(\sigma)}x_i^{-1} \, .
\]
In Section~\ref{sec:reciprocity}, we obtain the following reciprocity theorem.

\begin{theorem}\label{thm:reciprocity}
For any piecewise Laurent polynomial $f\in\PLP(\Sigma_n)$ and any matroid $\M$ on $[n]$, 
\[
Q_\M(f^\vee) \, = \, (-1)^{\rk(\M)-1} \, Q_\M(f\cdot\omega_\M)^\vee.
\]
\end{theorem}

In the special setting where $\M$ is the Boolean matroid and $f$ is the piecewise Laurent monomial
associated to a generalized permutohedron $P$ (i.e., $f$ is the collection of vertex monomials), Theorem~\ref{thm:reciprocity} precisely recovers
Ehrhart reciprocity for $P$ upon specializing the variables all to $1$. More generally, for any loopless matroid~$\M$,
Theorem~\ref{thm:reciprocity} specializes to give a combinatorial justification of Serre duality of the Euler characteristic $\chi_\M$, answering a question of Larson, Li, Payne, and Proudfoot~\cite{larsonetal}.

\subsection{Acknowledgements}

The authors thank Chris Eur, Matt Larson, and Sam Payne for feedback on early drafts. D.R. was partially supported by NSF Grant DMS-2302024.


\section{Prerequisites}\label{sec:prerequisites}

In this section, we collect necessary background in order to establish notation and conventions that set the stage for the precise statements and proofs of our main results.

\subsection{Braid fans and generalized permutohedra}

Let $E$ be a finite set of size $n$ and let $\{e_i:i\in E\}$ denote the standard basis vectors on $\R^E$. We will often, but not always, take $E=[n]$. For each subset $T\subseteq E$, let $e_T\coloneqq \sum_{i\in T}e_i$ denote the indicator vector associated to $T$. The \Def{braid fan} $\Sigma_E$ is the collection of polyhedral cones delineated by the hyperplanes in $\R^E$ where two coordinates are equal. The cones of the braid fan $\Sigma_E$ are indexed by ordered set partitions $E=T_1\sqcup\dots\sqcup T_m$, where we define the cone $\widetilde C_{T_1|\cdots|T_m}\subseteq N_\R$ as the set of vectors $b\in \R^E$ such that $b_i=b_j$ if $i,j\in T_k$ for some $k$ and $b_i\geq b_j$ if $i\in T_k$ and $j\in T_{\ell}$ with $k<\ell$. The dimension of a cone in $\Sigma_E$ is the number of nonempty parts in the corresponding ordered set partition. Notice that each cone in $\Sigma_E$ has lineality space spanned by $1=e_E$. In particular, 
\[
\widetilde C_{T_1|\cdots|T_m} \, = \, \R 1+\R_{\geq 0} \left( e_{T_1},e_{T_1\cup T_2},\dots,e_{T_1\cup
T_2\cup\dots\cup T_{m-1}} \right).
\]
Let $S_E$ denote the set of bijections $\sigma:[n]\rightarrow E$. Then the maximal cones of the braid fan are indexed by $S_E$---given a bijection $\sigma\in S_E$, we denote the corresponding maximal cone by
\[
\widetilde C_\sigma
 \, \coloneqq  \, \left\{b\in \R^E: \, b_{\sigma(1)}\geq \dots\geq b_{\sigma(n)} \right\}
 \,  = \, \widetilde C_{\sigma\{1\}|\sigma\{2\}|\cdots|\sigma\{n\}}.
\]
We are using the notation $\widetilde C_\sigma$ here to distinguish these cones from their duals $C_\sigma$ that
appeared in the introduction and will reappear below. Notice that $\widetilde C_\sigma$ and $\widetilde C_{\sigma'}$ meet along a codimension-one face if and only if $\sigma'=\sigma\circ\tau_i$ for some adjacent transposition $\tau_i\in S_n$ swapping $i$ and $i+1$; in this case, $\widetilde C_\sigma$ and $\widetilde C_{\sigma'}$ meet along the hyperplane $x_{\sigma(i)}=x_{\sigma(i+1)}$. Figure~\ref{fig:braidfan} depicts the projection of $\Sigma_3$ (i.e., $E=[3]$) by its lineality space, with each hyperplane labeled by its defining equation and the maximal cones labeled by permutations $\sigma\in S_3$ written in single-line notation.

\begin{figure}[ht]
\begin{tikzpicture}[scale=2]
\fill[green!20,opacity=.4] (0,0) circle [radius=1];
\draw[ultra thick, gray, opacity=0.5,](0,0) -- (1,0);
\draw[ultra thick, gray, opacity=0.5,] (0,0) -- (1/2,1.732/2);
\draw[ultra thick, gray, opacity=0.5,] (0,0) -- (-1/2,1.732/2);
\draw[ultra thick, gray, opacity=0.5,] (0,0) -- (-1,0);
\draw[ultra thick, gray, opacity=0.5,] (0,0) -- (1/2,-1.732/2);
\draw[ultra thick, gray, opacity=0.5,] (0,0) -- (-1/2,-1.732/2);
\node[above right] at (1/2,1.732/2) {\small $x_2=x_3$};
\node[above left] at (-1/2,1.732/2) {\small $x_1=x_3$};
\node[left] at (-1,0) {\small $x_1=x_2$};
\node at (.5,1.732/6) {\small $\widetilde C_{321}$};
\node at (0,1.732/3) {\small $\widetilde C_{231}$};
\node at (-.5,1.732/6) {\small $\widetilde C_{213}$};
\node at (.5,-1.732/6) {\small $\widetilde C_{312}$};
\node at (0,-1.732/3) {\small $\widetilde C_{132}$};
\node at (-.5,-1.732/6) {\small $\widetilde C_{123}$};
\end{tikzpicture}
\caption{The braid fan $\Sigma_3$.}\label{fig:braidfan}
\end{figure}
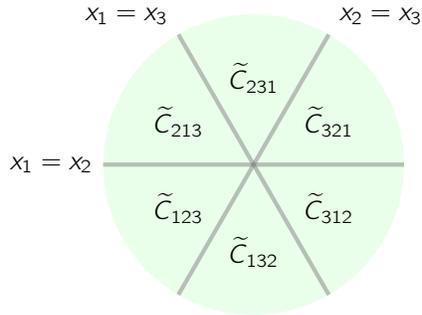

There are many ways to characterize generalized permutohedra \cite{ardilaetal,postnikovpermutahedra,postnikovreinerwilliams}. For our purposes, the most relevant characterization is that a lattice polytope $P\subset \R^E$ is a \Def{generalized permutohedron} if its (outward) normal fan $\Sigma_P$ coarsens the braid fan $\Sigma_E$.  Recall that, for each polytope $P\subset\R^E$ and each face $F$ of $P$, the \emph{affine face cone} of $F$ is $C_{P,F}\coloneqq a+\R_{\geq 0}(P-a)$, where $a$ is any element of the relative interior of $F$, and the \emph{outward normal cone} of $F$ is defined by
\[
\widetilde C_{P,F} \, \coloneqq  \, \left\{v\in\R^E:v\cdot w\leq 0\text{ for all }w\in C_{P,F}\right\}.
\]
The \emph{normal fan} $\Sigma_P$ of $P$ is the collection of all outward normal cones of nonempty faces of $F$, and
$P\subset \R^E$ is a generalized permutohedron precisely when each cone of the braid fan $\Sigma_E$ is contained in
some cone of the normal fan $\Sigma_P$. Figure~\ref{fig:normalfan} depicts a generalized permutohedron $P\subset\R^3$ and the projection of its normal fan by its lineality space.

\begin{figure}[ht]
\begin{tikzpicture}[scale=1.2]
\draw[thick,fill=green!20, fill opacity=.5] (1+1.732/2,1/2) -- (1-1.732/2,3/2) -- (1-1.732,1) -- (1-1.732,0) -- (1-1.732/2,-1/2) -- (1+1.732/2,1/2);

\node [right] at (1+1.732/2,1/2) {\tiny $v_1=(1,1,4)$};
\node [above] at (1-1.732/2,3/2) {\tiny $v_2=(1,3,2)$};
\node [above left] at (1-1.732,1) {\tiny $v_3=(2,3,1)$};
\node [below left] at (1-1.732,0) {\tiny$v_4=(3,2,1)$};
\node [below] at (1-1.732/2,-1/2) {\tiny $v_5=(3,1,2)$};
\node at (1/4,1/2) {$P$};
\end{tikzpicture}
\hspace{30bp}
\begin{tikzpicture}[scale=1.7]
\fill[green!20,opacity=.4] (0,0) circle [radius=1];
\draw[ultra thick, gray, opacity=0.5,] (0,0) -- (1/2,1.732/2);
\draw[ultra thick, gray, opacity=0.5,] (0,0) -- (-1/2,1.732/2);
\draw[ultra thick, gray, opacity=0.5,] (0,0) -- (-1,0);
\draw[ultra thick, gray, opacity=0.5,] (0,0) -- (1/2,-1.732/2);
\draw[ultra thick, gray, opacity=0.5,] (0,0) -- (-1/2,-1.732/2);
\node at (.5,0) {\small $\widetilde C_{P,v_1}$};
\node at (0,2/3) {\small $\widetilde C_{P,v_2}$};
\node at (-.6,1.732/6) {\small $\widetilde C_{P,v_3}$};
\node at (0,-2/3) {\small $\widetilde C_{P,v_5}$};
\node at (-.6,-1.732/6) {\small $\widetilde C_{P,v_4}$};
\end{tikzpicture}
\caption{A generalized permutohedron and its normal fan.}\label{fig:normalfan}
\end{figure}
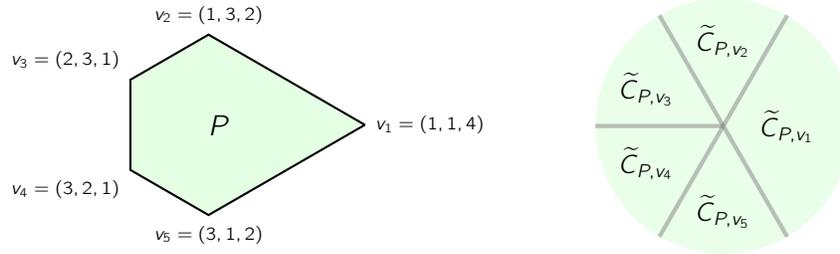

The maximal cones of $\Sigma_P$ are each of the form $\widetilde C_{P,v}$, where $v\in V(P)$ is a vertex of $P$. On the other hand, the maximal cones of the braid fan are indexed by bijections $\sigma\in S_E$. Thus, for every generalized permutohedron $P\subset\R^E$, there is a surjective function $v_P:S_E\rightarrow V(P)$, where $v_P(\sigma)$ is defined as the unique vertex of $P$ such that $\widetilde C_\sigma\subseteq \widetilde C_{P,v_P(\sigma)}$. Alternatively, given any vector $b\in \R^E$ with $b_{\sigma(1)}>\dots>b_{\sigma(n)}$, the vertex $v_P(\sigma)$ is the unique point of $P$ that has a maximal dot product with $b$. For the generalized permutohedron $P\subset\R^3$ depicted above, for example, $v_P(312)=v_P(321)=v_1$, while $v_P(123)=v_4$.

Lastly, we consider the case where $P\subset\R^E$ is a \emph{nondegenerate} generalized permutohedron, meaning that $\Sigma_P=\Sigma_E$. Then the function $v_P:S_n\rightarrow V(P)$ is a bijection, and the affine vertex cone $C_{P,v_P(\sigma)}$ is the translation of the outward normal cone of $\widetilde C_\sigma$ by $v_P(\sigma)$. Since the primitive outward normal vectors of $\widetilde C_\sigma$ are $e_{\sigma(2)}-e_{\sigma(1)},\dots,e_{\sigma(n)}-e_{\sigma(n-1)}$, it follows that 
\[
C_{P,v_P(\sigma)} \, = \, v_P(\sigma)+C_\sigma\;\;\;\text{ where }\;\;\;C_\sigma \, = \, \R_{\geq 0}(e_{\sigma(2)}-e_{\sigma(1)},\dots,e_{\sigma(n)}-e_{\sigma(n-1)}) \, ,
\]
as asserted in the introduction.

\subsection{Matroids} 

While there are many ways to characterize matroids \cite{oxley}, we choose one that is most useful for our purposes. A matroid $\M$ on a finite set $E$ is a nonempty collection of subsets of $E$, called \emph{bases},  which satisfy the basis-exchange property:
\[
\text{ for all }B_1,B_2\in \M \text{ and }i\in B_1\setminus B_2\text{, there exists } j\in B_2\setminus B_1\text{ such that }(B_1\setminus \{i\})\cup \{j\}\in \M.
\] 
We will often, but not always, take $E=[n]$.

Given a matroid $\M$ on $E$, we say that a subset $I\subseteq E$ is \emph{independent} if $I\subseteq B$ for some $B\in\M$.  The \emph{rank} $\rk_\M(T)$ of a subset $T\subseteq E$ is the size of the largest independent subset of $T$. The bases of $\M$ can be characterized as the maximal independent sets, and every basis has the same size, which is equal to the \emph{rank} of the matroid, denoted $\rk(\M)\coloneqq \rk_\M(E)$. 

An element $i\in E$ is a \emph{loop} of $\M$ if $i$ is not in any basis of $\M$, and $i$ is a \emph{coloop} of $\M$ if $i$ is in every basis of $\M$. Given a matroid $\M$ on $E$ and a subset $T\subseteq E$, the \emph{restriction} $\M|T$ is the matroid on $T$ whose bases are the maximal independent subsets of $T$, and the \emph{contraction} $\M/T$ is the matroid on $T^c$ whose bases are of the form $B\setminus T$ where $B$ is a basis of $\M$ that restricts to a basis of $\M|T$. The \emph{deletion} $\M\setminus T$ is the restriction of $\M$ to $T^c$.

The basis-exchange property implies that there is a surjection $B_\M:S_E\rightarrow\M$ where $B_\M(\sigma)\in\M$ is the \emph{greedy basis determined by $\sigma$} defined by
\[
B_\M(\sigma) \, \coloneqq  \, \big\{\sigma(i): \rk_\M\{\sigma(1),\dots,\sigma(i)\}>\rk_\M\{\sigma(1),\dots,\sigma(i-1)\}\big\}\, .
\]
Consider the \emph{base polytope} of $\M$, defined as the convex hull of the indicator vectors of bases:
\[
P_\M \, \coloneqq  \, \mathrm{conv} \left\{ e_B:B\in\M  \right\} \, \subset \, \R^E \, .
\]
The basis-exchange property implies that $e_{B_\M(\sigma)}\in P_\M$ is the unique element of $P_\M$ that has maximal dot product with any $b\in\R^E$ with $b_{\sigma(1)}>\dots>b_{\sigma(n)}$, and it follows that $P_M$ is a generalized permutohedron with $v_{P_\M}(\sigma)=e_{B_\M(\sigma)}$, as asserted in the introduction.

\subsection{Piecewise Laurent polynomials} 

Let $E$ be a set of size $n$. A function $f:\R^E\rightarrow\R$ is \emph{integral linear on $\R^E$} if it is linear and $f(\Z^E)\subseteq \Z$. A function $f:\R^E\rightarrow\R$ is called \emph{(integral) piecewise linear on $\Sigma_E$} if, for every cone $\widetilde C\in\Sigma_E$, the restriction $f|_{\widetilde C}$ is equal to the restriction of an integral linear function on $\R^E$. More specifically, a piecewise linear function is the data of an integral linear function $m_\sigma$ associated to each maximal cone $\widetilde C_\sigma\in\Sigma_E$ such that, for any adjacent transposition $\tau_i\in S_n$ swapping $i$ and $i+1$, the linear functions $m_\sigma$ and $m_{\sigma\circ\tau_i}$ agree along the hyperplane defined by $x_{\sigma(i)}=x_{\sigma(i+1)}$. Piecewise linear functions on $\Sigma_E$ form a group under addition, denoted $\PL(\Sigma_E)$. If $P\subset\R^E$ is a generalized permutohedron, then the collection of vertices $\{v_P(\sigma):\sigma\in S_E\}$ gives rise to a piecewise linear function $s_P\in \PL(\Sigma_E)$, called the \emph{support function of $P$}, where $(s_P)_\sigma:\R^E\rightarrow\R$ is the integral linear function defined by $(s_P)_\sigma(b)=v_P(\sigma)\cdot b$. 

For our purposes, it is more useful to work with a multiplicative version of the group of piecewise linear functions. For each integral linear function $m:\R^E\rightarrow\R$, define $x^m\coloneqq \prod_{i\in E}x_i^{m_i(e_i)}$. We say that a collection $\{x^{m_\sigma}:\sigma\in\Sigma_E\}$ is a \emph{piecewise Laurent monomial on $\Sigma_E$} if, for every $i\in[n-1]$, the Laurent polynomial $x^{m_\sigma}-x^{m_{\sigma\circ\tau_i}}$ is divisible by $x_{\sigma(i)}-x_{\sigma(i+1)}$. Piecewise Laurent monomials form a group under multiplication, denoted $\PLM(\Sigma_E)^*$ and there is a natural group isomorphism between $\PL(\Sigma_E)$ and $\PLM(\Sigma_E)^*$ that sends $f=\{m_\sigma:\sigma\in S_E\}$ to $x^f\coloneqq \{x^{m_\sigma}:\sigma\in S_E\}$. In particular, associated to each generalized permutohedron $P\subset\R^E$, we obtain a piecewise Laurent monomial $f_P=\{x^{v_P(\sigma)}:\sigma\in S_E\}\in\PLM(\Sigma_E)^*$.

More generally,  a collection $\{f_\sigma\in\Z[x_i^{\pm 1}:i\in E]:\sigma\in S_E\}$ is a \emph{piecewise Laurent polynomial on $\Sigma_E$} if, for every $i\in[n-1]$, the Laurent polynomial $f_\sigma-f_{\sigma\circ\tau_i}$ is divisible by $x_{\sigma(i)}-x_{\sigma(i+1)}$. Piecewise Laurent polynomials form a ring, denoted $\PLP(\Sigma_E)$, which contains piecewise Laurent monomials $\PLM(\Sigma_E)^*$ as the multiplicative subgroup of units. To close this subsection, we observe that piecewise Laurent polynomials on $\Sigma_E$ are spanned by the piecewise Laurent monomials.

\begin{lemma}\label{lemma:span}
$\PLM(\Sigma_E)^*$ spans $\PLP(\Sigma_E)$.
\end{lemma}

\begin{proof}
A piecewise linear function on $\Sigma_E$ is freely and uniquely determined by its values at $e_T$ with $\emptyset\subset T\subseteq E$. Given an integral linear function $m:\R^n\rightarrow\R$ and $\widetilde C\in\Sigma_E$, we let $m_{\widetilde C}\in\PL(\Sigma_E)$ denote the piecewise linear function determined by
\[
m_{\widetilde C}(e_T) \, = \, \begin{cases}
m(e_T) & e_T\in\widetilde C,\\
0 & \text{else.}
\end{cases}
\]
Similarly, for any Laurent polynomial $f=\sum_{i=1}^k a_ix^{m_i}\in\Z[x_i^{\pm 1}:i\in E]$, we define a piecewise Laurent polynomial $f_{\widetilde C}\in\PLP(\Sigma_E)$ by
\[
f_{\widetilde C} \, = \, \sum_{i=1}^k a_ix^{(m_i)_{\widetilde C}}.
\]
If $f=\{f_\sigma\}\in\PLP(\Sigma_E)$ and $\widetilde C\subseteq \widetilde C_{\sigma_1}\cap\widetilde C_{\sigma_2}$, then $(f_{\sigma_1})_{\widetilde C}=(f_{\sigma_2})_{\widetilde C}$, and we denote this element as $f_{\widetilde C}\in\PLP(\Sigma_E)$. Intuitively, each $f_{\widetilde C}$ is designed to agree with $f$ on $\widetilde C$ and to fade to the constant polynomial $\sum_{i=1}^k a_i$ on all cones that are not adjacent to $\widetilde C$. Notice that each $f_{\widetilde C}$ is in the span of $\PLM(\Sigma_E)^*$, and a straightforward inclusion-exclusion argument on the poset of cones in $\Sigma_E$ shows that $f$ can be written as a linear combination of $\{f_{\widetilde C}:\widetilde C\in\Sigma_E\}$.
\end{proof}


\section{Polynomiality}\label{sec:polynomiality}

In this section, we prove Theorem~\ref{thm:Polynomial}, which we state here in its most general form.

\begin{theorem}\label{thm:finite2}
Let $E$ be a set of size $n$. Let $\M$ be a matroid on $E$ and let $f\in \PLP(\Sigma_E)$ be a piecewise Laurent polynomial on the braid fan $\Sigma_E$. Define 
\[
Q_\M(f) \, \coloneqq  \sum_{\sigma\in S_E}f_\sigma \, Q(C_\sigma)\prod_{i\notin B_\M(\sigma)}(1-x_i)\;\;\;\text{ where
}\;\;\;Q(C_\sigma) \, \coloneqq  \, \frac{1}{\prod_{i=1}^{n-1}\big(1-\frac{x_{\sigma(i+1)}}{x_{\sigma(i)}}\big)} \, .
\]
Then $Q_\M(f)\in\Z[x_i^{\pm 1}:i\in E]$.
\end{theorem}

We now set up notation for our proof of Theorem~\ref{thm:finite2}. Without loss of generality, we assume $E=[n]$, and prove polynomiality inductively on $n$. Fix a permutation $\mu\in S_{n-1}$, which we write in single-line notation, and let $S_n(\mu)\subseteq S_n$ be the set of
permutations obtained from $\mu$ by inserting $n$ anywhere in $\mu$. Let $\mu_i$ denote the unique element of
$S_n(\mu)$ with $\mu_i(i)=n$ (in other words, the one where $n$ appears in the $i$th position).  For example,
if $\mu=123\in S_3$, then 
\[
S_4(\mu) \, = \, \left\{\mu_4=1234,\mu_3=1243,\mu_2=1423,\mu_1=4123\right\} .
\] 
As depicted in Figure~\ref{fig:quotient}, taking the quotient of $\R^n$ by $\R e_n$ naturally maps $\Sigma_n$ onto $\Sigma_{n-1}$, and $S_n(\mu)$ indexes the maximal cones in $\Sigma_n$ lying over the maximal cone in $\Sigma_{n-1}$ indexed by $\mu$. In terms of a generalized permutohedron $P$, the partitions $\mu_1,\dots,\mu_n$ trace a path of vertices beginning in the face where $x_n$ is maximal and ending in the face where $x_n$ is minimal.

\begin{figure}[ht]
\begin{tikzpicture}[scale=2]
\draw[ultra thick, gray, opacity=0.5](0,0) -- (0,1);
\node at (.2,.5) {$\widetilde C_\mu$};
\draw[ultra thick, gray, opacity=0.5] (0,0) -- (0,-1);
\node at (.2,-.5) {$\widetilde C_\nu$};
\node[gray, opacity=0.8] at (0,0) {$\bullet$};
\end{tikzpicture}
\hspace{40bp}
\begin{tikzpicture}[scale=2]
\fill[green!20,opacity=.4] (0,0) circle [radius=1];
\draw[ultra thick, gray, opacity=0.5,](0,0) -- (1,0);
\draw[ultra thick, gray, opacity=0.5,] (0,0) -- (1/2,1.732/2);
\draw[ultra thick, gray, opacity=0.5,] (0,0) -- (-1/2,1.732/2);
\draw[ultra thick, gray, opacity=0.5,] (0,0) -- (-1,0);
\draw[ultra thick, gray, opacity=0.5,] (0,0) -- (1/2,-1.732/2);
\draw[ultra thick, gray, opacity=0.5,] (0,0) -- (-1/2,-1.732/2);
\node at (.5,1.732/6) {\small $\widetilde C_{\mu_1}$};
\node at (0,1.732/3) {\small $\widetilde C_{\mu_2}$};
\node at (-.5,1.732/6) {\small $\widetilde C_{\mu_3}$};
\node at (.5,-1.732/6) {\small $\widetilde C_{\nu_1}$};
\node at (0,-1.732/3) {\small $\widetilde C_{\nu_2}$};
\node at (-.5,-1.732/6) {\small $\widetilde C_{\nu_3}$};
\end{tikzpicture}
\hspace{40bp}
\raisebox{8bp}{
\begin{tikzpicture}[scale=1.2]
\draw[thick,fill=green!20, fill opacity=.5] (1+1.732/2,1/2) -- (1-1.732/2,3/2) -- (1-1.732,1) -- (1-1.732,0) -- (1-1.732/2,-1/2) -- (1+1.732/2,1/2);
\node [right] at (1+1.732/2,1/2) {\tiny $v_P(\mu_1)=v_P(\nu_1)$};
\node [above] at (1-1.732/2,3/2) {\tiny $v_P(\mu_2)$};
\node [above left] at (1-1.732,1) {\tiny $v_P(\mu_3)$};
\node [below left] at (1-1.732,0) {\tiny$v_P(\nu_3)$};
\node [below] at (1-1.732/2,-1/2) {\tiny $v_P(\nu_2)$};
\node at (1/4,1/2) {$P$};
\end{tikzpicture}}
\caption{Geometric interpretations of $S_n(\mu)$.}\label{fig:quotient}
\end{figure}
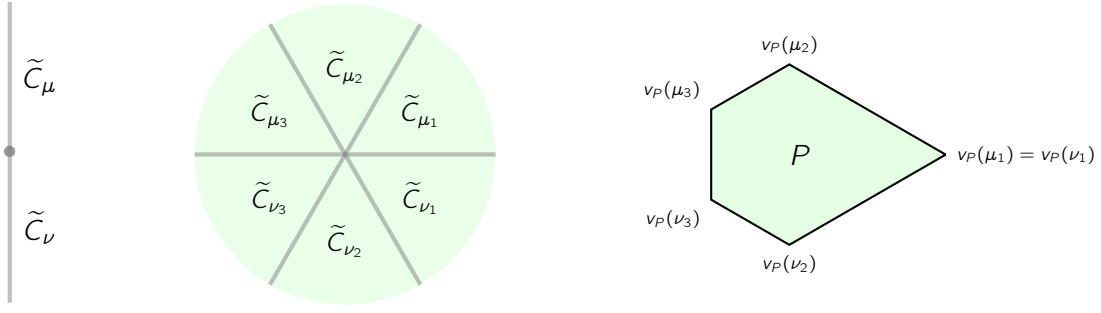

Given a piecewise Laurent polynomial $f=\{f_\sigma:\sigma\in S_n\}\in \PLP(\Sigma_n)$, we define Laurent polynomials $g_{\mu,j}\in\Z[x_1^{\pm 1},\dots,x_n^{\pm 1}]$ for $\mu\in S_{n-1}$ and $j=1,\dots,n-1$ by the formula
\[
g_{\mu,j} \, \coloneqq \, \frac{f_{\mu_j}-f_{\mu_{j+1}}}{x_n-x_{\mu(j)}}\, .
\]
The $g_{\mu,j}$ are well defined because $\mu_j=\mu_{j+1}\circ\tau_j$ with $\mu_{j}(j)=n$ while $\mu_{j}(j+1)=\mu(j)$. Furthermore, for each $\mu\in S_{n-1}$, we define the Laurent polynomial
\[
h_\mu \, \coloneqq  \, x_n \, g_{\mu,1}+\dots+x_n \, g_{\mu,n-1}+f_{\mu_n} \, .
\]
Finally, if $n$ is neither a loop nor a coloop of $\M$, then for each $\mu\in S_{n-1}$, we define
\[
k_\M(\mu) \, \coloneqq  \, \max\left\{i:\rk(\mu(1),\dots,\mu(i-1))<\rk(\mu(1),\dots,\mu(i-1),n)\right\},
\]
so that $k_\M(\mu)$ is the largest index $i\in[n-1]$ such that $n$ is in the greedy bases $B_\M(\mu_i)$.

With notation as above, Theorem~\ref{thm:finite2} follows from induction on $n$ and the following result.

\begin{proposition}\label{prop:indstep}
Let $f=\{f_\sigma:\sigma\in S_n\}\in \PLP(\Sigma_n)$ be a piecewise Laurent polynomial on $\Sigma_n$. For any $\mu\in S_{n-1}$, 
\[
\sum_{\sigma\in S_n(\mu)}f_\sigma \, Q(C_\sigma)\prod_{i\notin B_\M(\sigma)}(1-x_i) \, = \, f_\mu \,
Q(C_\mu)\prod_{i\notin B_{\M\setminus\{n\}}(\mu)}(1-x_i) \, ,
\]
where 
\begin{enumerate}[leftmargin=*]
\item the Laurent polynomials $f_\mu$ are given by the formulas
\[
\hspace{20bp}f_\mu=\begin{cases}
h_\mu  \hfill \text{if }n\text{ is a coloop of }\M,&\\
(1-x_n)h_\mu \hfill\text{if }n\text{ is a loop of }\M,\text{ and\phantom,} &\\
h_\mu-x_n \left( x_{\mu(k)}g_{\mu,1}+\dots+x_{\mu(k)} g_{\mu,k}+x_{\mu(k+1)}g_{\mu,k+1}+\dots+x_{\mu(n-1)}g_{\mu,n-1}
\right)\;\;\;\;\;\;\;\;\;\;\;\;\;\text{else,}&
\end{cases}
\]
where, in the third case, $k=k_\M(\mu)$;
\item the collection $\{f_\mu:\mu\in S_{n-1}\}$ is an element of $\PLP(\Sigma_{n-1})\otimes_\Z\Z[x_n^{\pm 1}]$.
\end{enumerate}
\end{proposition}

Before proving Proposition~\ref{prop:indstep}, we present two consequences that were mentioned above. The first consequence is that Prop~\ref{prop:indstep} provides a novel proof of Brion's Formula \eqref{eq:Brion} in the setting of generalized permutohedra, and its proof helps us familiarize ourselves with the polynomials $g_{\mu,i}$ and $h_\mu$ appearing in Proposition~\ref{prop:indstep}.

\begin{corollary}
If $P\subset\R^E$ is a generalized permutohedon, then
\[
\sum_{\sigma\in S_E}x^{v_P(\sigma)} Q(C_\sigma) \, = \, q(P) \, .
\]
\end{corollary}

\begin{proof}
Without loss of generality, suppose $E=[n]$. The formula in the left-hand side of the corollary is $Q_\M(f)$ where $\M$ is the Boolean matroid on $[n]$ and $f=f_P$. Since every element of the Boolean matroid is a coloop and since any deletion of a Boolean matroid is a Boolean matroid, it follows that we can study the left-hand side of the corollary recursively using the coloop case of Proposition~\ref{prop:indstep}:
\[
\sum_{\sigma\in S_n}x^{v_P(\sigma)} Q(C_\sigma) \, = \sum_{\mu\in S_{n-1}}h_\mu  \, Q(C_\mu) \, .
\]
It remains to understand $h_\mu$.

Since $P$ is a generalized permutohedron and $\mu_j=\mu_{j+1}\circ\tau_j$, it follows that $v_P(\mu_j)$ and $v_P(\mu_{j+1})$ are connected by an edge normal to the hyperplane $x_{\mu(j)}=x_n$, and since the $n$th coordinate of $v_P(\mu_j)$ is at least as big as the $n$th coordinate of $v_P(\mu_{j+1})$, it follows that
\[
v_P(\mu_{j+1})-v_P(\mu_j) \, = \, m(e_{\mu(j)}-e_n)\;\;\;\text{ for some }\;\;\;m\geq 0.
\]
Thus, we can compute $g_{\mu,j}$ explicitly:
\[
g_{\mu,j} \, = \, \frac{x^{v_P(\mu_j)}-x^{v_P(\mu_{j+1})}}{x_n-x_{\mu(j)}} \, = \,
\frac{x^{v_P(\mu_j)}}{x_n}\left(1+\frac{x_{\mu(j)}}{x_n}+\dots+\left(\frac{x_{\mu(j)}}{x_n}\right)^{m-1}\right) .
\]
In particular, notice that $x_ng_{\mu,j}$ is the Laurent polynomial of lattice points in the half-open edge
$\big[v_P(\mu_j),v_P(\mu_{j+1})\big)$, depicted in Figure~\ref{fig:latticepoints}.

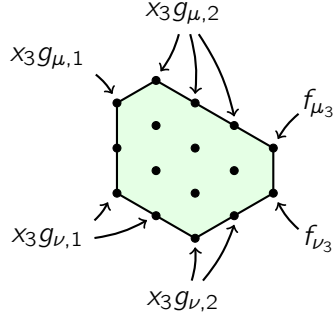
\begin{figure}
\begin{tikzpicture}[scale=1.2]
\draw[thick,fill=green!20, fill opacity=.5] (1,0) -- (1,1/2) -- (1-3*1.732/4,5/4) -- (1-1.732,1) -- (1-1.732,0) -- (1-1.732/2,-1/2) -- (1,0);

\node (m1) at (-1.5,1.5) {$x_3g_{\mu,1}$};
\node (m2) at (0,2) {$x_3g_{\mu,2}$};
\node (m3) at (1.5,1) {$f_{\mu_3}$};

\draw[thick,->, bend left=10] (m1) to (1-1.732-.05,1+.1);

\draw[thick,->, bend left=10] (m2) to (1-3*1.732/4+.05,5/4+.1);
\draw[thick,->, bend left=10] (m2) to (1-1.732/2,1+.1);
\draw[thick,->, bend left=10] (m2) to (1-1.732/4,3/4+.1);

\draw[thick,->, bend right=10] (1.5-.22,1-.12) to (1+.05,1/2+.1);

\node (n1) at (-1.5,-.5) {$x_3g_{\nu,1}$};
\node (n2) at (0,-1.2) {$x_3g_{\nu,2}$};
\node (n3) at (1.5,-.5) {$f_{\nu_3}$};

\draw[thick,->, bend right=10] (n1) to (1-1.732-.1,-.1);
\draw[thick,->, bend right=10] (n1) to (1-3*1.732/4-.1,-1/4-.05);

\draw[thick,->, bend right=10] (n2) to (1-1.732/2,-1/2-.1);
\draw[thick,->, bend right=10] (n2) to (1-1.732/4,-1/4-.1);

\draw[thick,->, bend left=10] (1.5-.22,-.5+.1) to (1.05,-.1);

\node at (1,0) [draw, shape=circle, fill=black, minimum size=3pt, inner sep=0pt] {};
\node at (1,1/2) [draw, shape=circle, fill=black, minimum size=3pt, inner sep=0pt] {};
\node at (1-3*1.732/4,5/4) [draw, shape=circle, fill=black, minimum size=3pt, inner sep=0pt] {};
\node at (1-1.732,1) [draw, shape=circle, fill=black, minimum size=3pt, inner sep=0pt] {};
\node at (1-1.732,1/2) [draw, shape=circle, fill=black, minimum size=3pt, inner sep=0pt] {};
\node at (1-1.732,0) [draw, shape=circle, fill=black, minimum size=3pt, inner sep=0pt] {};
\node at (1-3*1.732/4,-1/4) [draw, shape=circle, fill=black, minimum size=3pt, inner sep=0pt] {};
\node at (1-1.732/2,-1/2) [draw, shape=circle, fill=black, minimum size=3pt, inner sep=0pt] {};
\node at (1-1.732/4,-1/4) [draw, shape=circle, fill=black, minimum size=3pt, inner sep=0pt] {};

\node at (1-1.732/2,1/2) [draw, shape=circle, fill=black, minimum size=3pt, inner sep=0pt] {};

\node at (1-1.732/4,1/4) [draw, shape=circle, fill=black, minimum size=3pt, inner sep=0pt] {};
\node at (1-1.732/4,3/4) [draw, shape=circle, fill=black, minimum size=3pt, inner sep=0pt] {};
\node at (1-1.732/2,1) [draw, shape=circle, fill=black, minimum size=3pt, inner sep=0pt] {};
\node at (1-3*1.732/4,3/4) [draw, shape=circle, fill=black, minimum size=3pt, inner sep=0pt] {};
\node at (1-3*1.732/4,1/4) [draw, shape=circle, fill=black, minimum size=3pt, inner sep=0pt] {};
\node at (1-1.732/2,0) [draw, shape=circle, fill=black, minimum size=3pt, inner sep=0pt] {};
\end{tikzpicture}
\caption{Identifying lattice points contributing to $h_\mu$.}\label{fig:latticepoints}
\end{figure}

Therefore, we see that $h_\mu=x_ng_{\mu,1}+\dots+x_ng_{\mu,n-1}+f_{\mu_n}$ is the Laurent polynomial of lattice points passed when traversing along edges in the positive $e_n$ direction from $v_P(\mu_1)$ to $v_P(\mu_n)$. It follows that, as elements of $\PLP(\Sigma_{n-1})\otimes_\Z\Z[x_n^{\pm 1}]$, 
\[
\{h_\mu:\mu\in S_{n-1}\} \, = \, \sum_{k\in\Z} x_n^kf_{P_k} \, ,
\]
where $P_k\subset\R^{n-1}$ is the projection of the slice $P\cap\{x_n=k\}$ along the quotient map $\R^n\rightarrow\R^{n-1}$. Thus, by induction, we conclude that
\[
\sum_{\sigma\in S_n}x^{v_P(\sigma)} Q(C_\sigma) \, = \, \sum_{k\in\Z}x_n^kq(P_k) \, = \, q(P) \, .\qedhere
\]
\end{proof}

The previous result described the Laurent polynomial $Q_\M(P)$ when $\M$ was the Boolean matroid and $P$ was any generalized permutohedron. The next result, on the other hand, allows $\M$ to be any matroid but restricts $P$ to the trivial generalized permutohedron $P=\{0\}$.

\begin{corollary}\label{corollary:basecase}
For any matroid $\M$ on $E$,
\[
Q_\M(1) \, =\prod_{i\in E\atop i\text{ is a loop of }\M}(1-x_i) \, .
\]
\end{corollary}

\begin{proof}
Without loss of generality, suppose $E=[n]$. Since $f=1$, it follows that $g_{\mu,j}=0$ for all $\mu\in S_{n-1}$ and $j=1,\dots,n-1$, and so $h_\mu=f_{\mu_n}=1$. It then follows from Proposition~\ref{prop:indstep} that 
\[
h_\mu=\begin{cases}
1 & \text{ if } n\text{ is not a loop,}\\
(1-x_n) & \text{ if } n\text{ is a loop.}
\end{cases}
\]
Since $i$ is a loop of $\M$ if and only if $i$ is a loop of $\M\setminus\{n\}$, the result follows by induction on $n$.
\end{proof}

The rest of this section is devoted to a proof of Proposition~\ref{prop:indstep}. Our first lemma gives explicit formulas for partial sums of the rational generating functions associated to the cones $C_{\mu_1},\dots,C_{\mu_n}$.

\begin{lemma}\label{lem:cones}
For any $\mu\in S_{n-1}$ and $1\leq i\leq n-1$, 
\[
Q(C_{\mu_1})+\dots+Q(C_{\mu_i}) \, = \, \frac{-\frac{x_n}{x_{\mu(i)}}}{\big(1-\frac{x_{\mu(2)}}{x_{\mu(1)}}\big)\cdots\big(1-\frac{x_{\mu(n-1)}}{x_{\mu(n-2)}}\big)\big(1-\frac{x_{n}}{x_{\mu(i)}}\big)}
\]
and
\[
Q(C_{\mu_{i+1}})+\dots+Q(C_{\mu_n}) \, = \,
\frac{1}{\big(1-\frac{x_{\mu(2)}}{x_{\mu(1)}}\big)\cdots\big(1-\frac{x_{\mu(n-1)}}{x_{\mu(n-2)}}\big)\big(1-\frac{x_{n}}{x_{\mu(i)}}\big)}
\, .
\]
Adding, we obtain
\[
Q(C_{\mu_1})+\dots+Q(C_{\mu_n}) \, = \,
\frac{1}{\big(1-\frac{x_{\mu(2)}}{x_{\mu(1)}}\big)\cdots\big(1-\frac{x_{\mu(n-1)}}{x_{\mu(n-2)}}\big)} \, = \, Q(C_\mu)
\, .
\]
\end{lemma}

\begin{proof}
Using the explicit formulas for $Q(C_{\mu_i})$, the first partial sum can be verified by induction on $i$, while the second can be verified by descending induction on $i$.
\end{proof}

Our next lemma describes the relationship between the greedy bases $B_\M(\mu_i)$ and $B_{\M\setminus\{n\}}(\mu)$.

\begin{lemma}\label{lem:bases}
If $n$ is a loop of $\M$, then $B_\M(\mu_i)=B_{\M\setminus\{n\}}(\mu)$ for every $i\in[n]$. If $n$ is a coloop of $\M$, then $B_\M(\mu_i)=B_{\M\setminus\{n\}}(\mu)\cup \{n\}$ for every $i\in[n]$. If $n$ is neither a loop nor a coloop of $M$, then
\[
B_\M(\mu_i) \, = \, \begin{cases}
(B_{\M\setminus\{n\}}(\mu)\setminus \{\mu(k)\})\cup \{n\} & \text{ if } i\leq k,\\
B_{\M\setminus\{n\}}(\mu) & \text{ if } i>k,
\end{cases}
\]
where $k=k_\M(\mu)$.
\end{lemma}

\begin{proof}
By definition of the greedy basis $B_\M(\mu_i)$, the only elements where $B_\M(\mu_i)$ and $B_\M(\mu_{i+1})$ can differ are $\mu(i)$ and $n$, and since all bases are the same size, it follows that either $B_\M(\mu_i)=B_\M(\mu_{i+1})$ or $B_\M(\mu_i)=(B_\M(\mu_{i+1})\setminus \{\mu(i)\})\cup \{n\}$. If $n$ is a loop, then $n$ cannot be in any basis, so all of the bases $B_\M(\mu_i)$ are the same subset of $[n-1]$, and by definition of the deletion matroid, this subset is $B_{\M\setminus\{n\}}(\mu)$. Similarly, if $n$ is a coloop, then every basis contains $n$, so all of the bases $B_\M(\mu_i)$ are the same subset of $[n]$, and by definition of the deletion matroid, this common basis is equal to $B_{\M\setminus\{n\}}(\mu)\cup \{n\}$. 

Now suppose that $n$ is neither a loop nor a coloop of $\M$. By definition of $k=k_\M(\mu)$, we have that $n\in B_M(\mu_k)$ and $n\notin B_\M(\mu_{k+1})$, so
\[
B_\M(\mu_k) \, = \, (B_\M(\mu_{k+1})\setminus \{\mu(k)\})\cup \{n\} \, .
\]
Furthermore, none of the bases $B_\M(\mu_{k+1}),\dots,B_\M(\mu_n)$ contain $n$, so it follows that these bases are all equal. Lastly, given that $\rk(\mu(1),\dots,\mu(k-1),n)>\rk(\mu(1),\dots,\mu(k-1))$, we see that $n$ is a coloop of the restriction matroid $\M|{\{\mu(1),\dots,\mu(k-1),n\}}$, so $n\in B_\M(\mu_i)$ for any $i=1,\dots,k$, implying that $B_\M(\mu_1)=\dots=B_\M(\mu_k)$. The lemma is then proved upon observing that, by definition of the deletion matroid, $B_{\M\setminus\{n\}}(\mu)=B_\M(\mu_n)$.
\end{proof} 

We are now prepared to prove the first part of Proposition~\ref{prop:indstep}.

\begin{proof}[Proof of Proposition~\ref{prop:indstep} (1)]
First suppose that $n$ is neither a loop nor a coloop of $\M$. Fix $\mu$ and let $k=k_\M(\mu)$. Then
\[
\sum_{j=1}^nf_{\mu_j} Q(C_{\mu_j})\prod_{i\notin B_\M(\mu_j)}(1-x_i) \, =\prod_{i\notin
B_\M(\mu_1)}(1-x_i)\sum_{j=1}^{k}f_{\mu_j} Q(C_{\mu_j})+\prod_{i\notin B_\M(\mu_n)}(1-x_i)\sum_{j=k+1}^nf_{\mu_j}
Q(C_{\mu_j}) \, .
\]
Iterating the identity
\[
f_{\mu_j} \, = \, (x_n-x_{\mu(j)})g_{\mu,j}+f_{\mu_{j+1}} \, ,
\]
we obtain, for $j=1,\dots,k-1$, 
\begin{equation}\label{eq:iteratedidentity}
f_{\mu_j} \, = \, (x_n-x_{\mu(j)})g_{\mu,j}+\dots+(x_n-x_{\mu(k-1)})g_{\mu,k-1}+f_{\mu_k}\, .
\end{equation}
Thus,
\begin{align}
\nonumber \sum_{j=1}^{k}f_{\mu_j} Q(C_{\mu_j}) \, & = \, f_{\mu_k}\sum_{j=1}^{k} Q(C_{\mu_j})+\sum_{i=1}^{k-1}(x_n-x_{\mu(i)})g_{\mu,i}\sum_{\ell=1}^iQ(C_{\mu_\ell})\\
&= \,
\frac{-\frac{x_{n}}{x_{\mu(k)}}f_{\mu_k}}{\big(1-\frac{x_{\mu(2)}}{x_{\mu(1)}}\big)\cdots\big(1-\frac{x_{\mu(n-1)}}{x_{\mu(n-2)}}\big)\big(1-\frac{x_{n}}{x_{\mu(k)}}\big)}+\frac{x_n
(g_{\mu,1}+\dots+g_{\mu,k-1})}{\big(1-\frac{x_{\mu(2)}}{x_{\mu(1)}}\big)\cdots\big(1-\frac{x_{\mu(n-1)}}{x_{\mu(n-2)}}\big)}
\, , \label{eq:big1}
\end{align}
where the second equality uses Lemma~\ref{lem:cones} and requires some simplification of rational functions. Similarly, for $j=k+1,\dots,n$,
\[
f_{\mu_j} \, = \, f_{\mu_{k}}-(x_n-x_{\mu(k)})g_{\mu,k}-\dots-(x_n-x_{\mu(j-1)})g_{\mu,j-1} \, ,
\]
from which we obtain
\begin{align}
\nonumber \sum_{j=k+1}^nf_{\mu_j} Q(C_{\mu_j}) \, &= \, f_{\mu_k}\sum_{j=k+1}^nQ(C_{\mu_j})-\sum_{i=k}^{n-1}(x_n-x_{\mu(i)})g_{\mu,i}\sum_{\ell=i+1}^nQ(C_{\mu_j})\\
&= \,
\frac{f_{\mu_k}}{\big(1-\frac{x_{\mu(2)}}{x_{\mu(1)}}\big)\cdots\big(1-\frac{x_{\mu(n-1)}}{x_{\mu(n-2)}}\big)\big(1-\frac{x_{n}}{x_{\mu(k)}}\big)}+\frac{x_{\mu(k)}g_{\mu,k}+\dots+x_{\mu(n-1)}g_{\mu,n-1}}{\big(1-\frac{x_{\mu(2)}}{x_{\mu(1)}}\big)\cdots\big(1-\frac{x_{\mu(n-1)}}{x_{\mu(n-2)}}\big)}
\, . \label{eq:big2}
\end{align}
Pulling everything together, we compute
\begin{align*}
\eqref{eq:big1}\cdot \prod_{i\notin B_\M(\mu_1)}(1-x_i)&+\eqref{eq:big2}\cdot \prod_{i\notin B_\M(\mu_n)}(1-x_i)\\
&= \, \big(\eqref{eq:big1}\cdot(1-x_{\mu(k)})+\eqref{eq:big2}\cdot(1-x_n)\big)\cdot\prod_{i\notin B_{\M\setminus\{n\}}(\mu)}(1-x_i)\\
&= \, f_\mu Q(C_\mu)\prod_{i\notin B_{\M\setminus\{n\}}(\mu)}(1-x_i) \, ,
\end{align*}
where
\begin{align*}
f_\mu \, &= \, f_{\mu_k}+(1-x_{\mu(k)})(x_n g_{\mu,1}+\dots+x_n g_{\mu,k-1})+(1-x_n)(x_{\mu(k)}g_{\mu,k}+\dots+x_{\mu(n-1)}g_{\mu,n-1})\\
&= \, h_\mu-x_n(x_{\mu(k)}g_{\mu,1}+\dots+x_{\mu(k)} g_{\mu,k}+x_{\mu(k+1)}g_{\mu,k+1}+\dots+x_{\mu(n-1)}g_{\mu,n-1})
\, ,
\end{align*}
with the last equality following from \eqref{eq:iteratedidentity}.This proves Proposition~\ref{prop:indstep} (1) in the case where $n$ is neither a loop nor a coloop. If $n$ is a loop, then we can adapt the above arguments to obtain
\[
\sum_{j=1}^nf_{\mu_j} Q(C_{\mu_j})\prod_{i\notin B_\M(\mu_j)}(1-x_i) \, = \, \sum_{j=1}^nf_{\mu_j} Q(C_{\mu_j})(1-x_n)\prod_{i\notin B_{\M\setminus\{n\}}(\mu)}(1-x_i)
\]
where
\[
\sum_{j=1}^nf_{\mu_j} Q(C_{\mu_j}) \, = \,
f_{\mu_k}\sum_{j=k+1}^nQ(C_{\mu_j})-\sum_{i=k}^{n-1}(x_n-x_{\mu(i)})g_{\mu,i}\sum_{\ell=i+1}^nQ(C_{\mu_j}) \, = \, h_\mu
Q(C_\mu) \, ,
\]
proving the loop case of Proposition~\ref{prop:indstep} (1). And if $n$ is a coloop, we similarly obtain
\[
\sum_{j=1}^nf_{\mu_j} Q(C_{\mu_j})\prod_{i\notin B_\M(\mu_j)}(1-x_i) \, = \, \sum_{j=1}^nf_{\mu_j}
Q(C_{\mu_j})\prod_{i\notin B_{\M\setminus\{n\}}(\mu)}(1-x_i) \, = \, h_\mu Q(C_\mu)\prod_{i\notin
B_{\M\setminus\{n\}}(\mu)}(1-x_i) \, ,
\]
yielding the coloop case of Proposition~\ref{prop:indstep} (1).
\end{proof}

To prove the second part of Proposition~\ref{prop:indstep}, we need to compare the polynomials $f_\mu$ and $f_{\mu\circ\tau_i}$. The next lemma will be useful.

\begin{lemma}\label{lem:divides}
Let $\mu\in S_{n-1}$ and $\mu'=\mu\circ\tau_i$ for some $i\in[n-2]$.
\begin{enumerate}
\item If $j\in[n-1]\setminus\{i, i+1\}$, then $x_{\mu(i+1)}-x_{\mu(i)}$ divides $g_{\mu,j}-g_{\mu',j}$.
\item $x_{\mu(i+1)}-x_{\mu(i)}$ divides $(g_{\mu,i}+g_{\mu,i+1})-(g_{\mu',i}+g_{\mu',i+1})$.
\end{enumerate}
\end{lemma}
\begin{proof}
Recall that $g_{\mu,j}=\frac{f_{\mu_j}-f_{\mu_{j+1}}}{x_n-x_{\mu(j)}}$. First, assume $j\notin\{i,i+1\}$. Then $\mu'(j)=\mu(j)$, so
\[
g_{\mu,j}-g_{\mu',j} \, = \,
\frac{f_{\mu_j}-f_{\mu_{j+1}}}{x_n-x_{\mu(j)}}-\frac{f_{\mu'_j}-f_{\mu'_{j+1}}}{x_n-x_{\mu'(j)}}=\frac{f_{\mu_j}-f_{\mu'_j}}{x_n-x_{\mu(j)}}+\frac{f_{\mu'_{j+1}}-f_{\mu_{j+1}}}{x_n-x_{\mu(j)}}
\, .
\]
Notice that both pairs of permutations $(\mu_j,\mu_j')$ and $(\mu_{j+1},\mu_{j+1}')$ differ by swapping the adjacent elements $\mu(i)$ and $\mu(i+1)$, and since $f$ forms a piecewise Laurent polynomial on $\Sigma_n$, both numerators in the previous expression are divisible by $x_{\mu(i+1)}-x_{\mu(i)}$, proving that $x_{\mu(i+1)}-x_{\mu(i)}$ divides $g_{\mu,j}-g_{\mu',j}$. 

For the second case, notice that $\mu'(i)=\mu(i+1)$ and $\mu'(i+1)=\mu_i(i)$, from which we compute
\begin{align*}
(g_{\mu,i}+g_{\mu,i+1})-(g_{\mu',i}+g_{\mu',i+1}) \, &= \, \frac{f_{\mu_{i}}-f_{\mu_{i+1}}}{x_n-x_{\mu(i)}}+\frac{f_{\mu_{i+1}}-f_{\mu_{i+2}}}{x_n-x_{\mu(i+1)}}-\frac{f_{\mu'_{i}}-f_{\mu'_{i+1}}}{x_n-x_{\mu'(i)}}-\frac{f_{\mu'_{i+1}}-f_{\mu'_{i+2}}}{x_n-x_{\mu'(i+1)}}\\
&= \,
\frac{f_{\mu_{i}}-f_{\mu_{i+1}}}{x_n-x_{\mu(i)}}+\frac{f_{\mu_{i+1}}-f_{\mu_{i+2}}}{x_n-x_{\mu(i+1)}}-\frac{f_{\mu'_{i}}-f_{\mu'_{i+1}}}{x_n-x_{\mu(i+1)}}-\frac{f_{\mu'_{i+1}}-f_{\mu'_{i+2}}}{x_n-x_{\mu(i)}}
\, .
\end{align*}
Upon specializing to $x_{\mu(i+1)}=x_{\mu(i)}$, this simplifies to
\[
\frac{f_{\mu_{i}}-f_{\mu'_{i}}}{x_n-x_{\mu(i)}}+\frac{f_{\mu'_{i+2}}-f_{\mu_{i+2}}}{x_n-x_{\mu(i)}} \, .
\]
Since $\mu_i'=\mu_i\circ\tau_{i+1}$ with $\mu_i(i+1)=\mu(i)$ and $\mu_i(i+2)=\mu(i+1)$, it follows that the first numerator vanishes when evaluated at $x_{\mu(i+1)}=x_{\mu(i)}$. And since $\mu_{i+2}'=\mu_{i+2}\circ\tau_i$ with $\mu_{i+2}(i)=\mu(i)$ and $\mu_{i+2}(i+1)=\mu(i+1)$, it follows that the second numerator vanishes when evaluated at $x_{\mu(i+1)}=x_{\mu(i)}$. This proves that  $x_{\mu(i+1)}-x_{\mu(i)}$ divides $(g_{\mu,i}+g_{\mu,i+1})-(g_{\mu',i}+g_{\mu',i+1})$.
\end{proof}

The next lemma explains how the values $k_\M(\mu)$ and $k_\M(\mu\circ\tau_i)$ are related.
\begin{lemma}\label{lemma:thresholds}
Let $\M$ be a matroid on $[n]$ in which $n$ is neither a loop nor a coloop, and let $\mu\in S_{n-1}$. Set $\mu'=\mu\circ\tau_i$ for some $i\in[n-2]$, and set $k=k_\M(\mu)$ and $k'=k_\M(\mu')$. If $k\in[n-1]\setminus\{i,i+1\}$, then $k'=k$, and if $k\in\{i,i+1\}$, then $k'\in\{i,i+1\}$.
\end{lemma}

\begin{proof}
Begin by noticing that $k\in[n-1]$ is uniquely determined by the fact that $n$ is a coloop of $\M|{\{\mu(1),\dots,\mu(k-1),n\}}$ but not a coloop of $\M|{\{\mu(1),\dots,\mu(k),n\}}$. If either $k<i$ or $k>i+1$, then 
\[
\{\mu(1),\dots,\mu(k-1)\}=\{\mu'(1),\dots,\mu'(k-1)\}\;\;\;\text{ and
}\;\;\;\{\mu(1),\dots,\mu(k)\}=\{\mu'(1),\dots,\mu'(k)\} \, , 
\]
and it follows that $k=k'$. It remains to consider the two cases $k\in\{i,i+1\}$.

If $k=i$, then $n$ is a coloop of $\M|{\{\mu(1),\dots,\mu(i-1),n\}}$ but not a coloop of
$\M|{\{\mu(1),\dots,\mu(i),n\}}$, so $n$ is a coloop of $\M|{\{\mu'(1),\dots,\mu'(i-1),n\}}$ but not a coloop of
$\M|{\{\mu'(1),\dots,\mu'(i),\mu'(i+1),n\}}$, and so the threshold value $k'$ must be in $\{i,i+1\}$.

Similarly, if $k=i+1$, then $n$ is a coloop of $\M|{\{\mu(1),\dots,\mu(i),n\}}$ but $n$ is not a coloop of $\M|{\{\mu(1),\dots,\mu(i+1),n\}}$. It follows that $n$ is a coloop of $\M|{\{\mu'(1),\dots,\mu'(i-1),n\}}$ and is not a coloop of $\M|{\{\mu'(1),\dots,\mu'(i),\mu'(i+1),n\}}$, so we must have $k'\in\{i,i+1\}$.
\end{proof}

We now prove the second part of the Proposition~\ref{prop:indstep}.

\begin{proof}[Proof of Proposition~\ref{prop:indstep} (2)]
Let $\mu\in S_{n-1}$ and set $\mu'=\mu\circ\tau_i$ for some $i\in[n-2]$. We must prove that $x_{\mu(i+1)}-x_{\mu(i)}$ divides $f_\mu-f_{\mu'}.$ Since $\mu'_n=\mu_n\circ\tau_i$, it follows that $f_{\mu_n}-f_{\mu'_n}$ is divisible by $x_{\mu(i+1)}-x_{\mu(i)}$, and Lemma~\ref{lem:divides} also implies that $x_{\mu(i+1)}-x_{\mu(i)}$ divides
\[
(g_{\mu,1}+\dots+g_{\mu,n-1})-(g_{\mu',1}+\dots+g_{\mu',n-1}) \, .
\]
Thus, $x_{\mu(i+1)}-x_{\mu(i)}$ divides $h_\mu-h_{\mu'}$. In particular, this proves Proposition~\ref{prop:indstep} in the case where $n$ is either a loop or a coloop of $\M$.

Now suppose that $n$ is neither a loop nor a coloop. Set $k=k_\M(\mu)$ and $k'=k_\M(\mu')$. If $k>i+1$, then $k=k'$ by Lemma~\ref{lemma:thresholds}, and $\mu(k)=\mu'(k)$. Thus, the formulas for $f_\mu$ and $f_{\mu'}$ imply that
\begin{align}\label{eq:difference}
f_\mu-f_{\mu'} \, = \, (h_\mu-h_{\mu'})&-x_{\mu(k)}x_n((g_{\mu,1}+\dots+g_{\mu,k-1})-(g_{\mu',1}+\dots+g_{\mu',k-1}))\\
\nonumber &-x_n(x_{\mu(k)}(g_{\mu,k}-g_{\mu',k})+\dots+x_{\mu(n-1)}(g_{\mu,n-1}-g_{\mu',n-1}) \, .
\end{align}
Since $k>i+1$, it then follows from Lemma~\ref{lem:divides} that $x_{\mu(i+1)}-x_{\mu(i)}$ divides every term on the right-hand side. If $k<i$, then we still have $k=k'$ and $\mu(k)=\mu'(k)$, and the only thing that changes in the expression \eqref{eq:difference} for $f_\mu-f_{\mu'}$ is that we replace 
\[
x_{\mu(i)}(g_{\mu,i}-g_{\mu',i})+x_{\mu(i+1)}(g_{\mu,i+1}-g_{\mu',i+1})\;\;\;\text{ with
}\;\;\;x_{\mu(i)}(g_{\mu,i}-g_{\mu',i+1})+x_{\mu(i+1)}(g_{\mu,i+1}-g_{\mu',i}) \, .
\]
However, after specializing at $x_{\mu(i)}=x_{\mu(i+1)}$, Lemma~\ref{lem:divides} implies that both of these terms are zero, and it then follows that $x_{\mu(i+1)}-x_{\mu(i)}$ divides $f_\mu-f_{\mu'}$ in this case, as well.

If $k=k'=i$, then the expression \eqref{eq:difference} for $f_\mu-f_{\mu'}$ must be adapted by replacing
\[
x_{\mu(i)}(g_{\mu',1}+\dots+g_{\mu',i-1})\;\;\;\text{ with }\;\;\;x_{\mu(i+1)}(g_{\mu',1}+\dots+g_{\mu',i-1})
\]
and by replacing 
\[
x_{\mu(i)}(g_{\mu,i}-g_{\mu',i})+x_{\mu(i+1)}(g_{\mu,i+1}-g_{\mu',i+1})\;\;\;\text{ with
}\;\;\;x_{\mu(i)}(g_{\mu,i}-g_{\mu',i+1})+x_{\mu(i+1)}(g_{\mu,i+1}-g_{\mu',i}) \, .
\]
Upon specializing at $x_{\mu(i)}=x_{\mu(i+1)}$, Lemma~\ref{lem:divides} implies that the differences of each of these replacements vanish, so $x_{\mu(i+1)}-x_{\mu(i)}$ divides  $f_\mu-f_{\mu'}$. If $k=k'=i+1$, on the other hand, then the expression \eqref{eq:difference} for $f_\mu-f_{\mu'}$ must be adapted by replacing
\[
x_{\mu(i+1)}(g_{\mu',1}+\dots+g_{\mu',i})\;\;\;\text{ with }\;\;\;x_{\mu(i)}(g_{\mu',1}+\dots+g_{\mu',i})
\]
and by replacing 
\[
-x_{\mu(i+1)}(g_{\mu,i+1}-g_{\mu',i+1})\;\;\;\text{ with }\;\;\;-x_{\mu(i+1)}g_{\mu,i+1}-x_{\mu(i)}g_{\mu',i+1} \, .
\]
Again, upon specializing at $x_{\mu(i)}=x_{\mu(i+1)}$, Lemma~\ref{lem:divides} implies that the difference of the sums of these replacements vanish, so $x_{\mu(i+1)}-x_{\mu(i)}$ divides  $f_\mu-f_{\mu'}$.

Finally, consider when $k=i$ and $k'=i+1$. Notice that $\mu'(k')=\mu'(i+1)=\mu(i)=\mu(k)$. In this case, the expression \eqref{eq:difference} for $f_\mu-f_{\mu'}$ must be adapted by adding $x_{\mu(i)}x_ng_{\mu',i}$ to the first line and subtracting the same quantity from the second line, so $x_{\mu(i+1)}-x_{\mu(i)}$ divides  $f_\mu-f_{\mu'}$ in this final case, as well, finishing the proof.
\end{proof}

%


\section{Recursion}\label{sec:recursion}

The primary goals of this section are to prove Theorem~\ref{thm:recursion} and then to relate the Laurent polynomials
$Q_\M(f)$ to matroid Euler characteristics. We begin by establishing precise notation for Theorem~\ref{thm:recursion},
using Figure~\ref{fig:slidingfacet} from the introduction, reproduced in Figure~\ref{fig:slidingfacetagain}, as a conceptual aid.

\begin{figure}[ht]
\begin{center}
\begin{tikzpicture}[scale=1.2]
\draw[thick,fill=green!20, fill opacity=.5] (1,0) -- (1,1) -- (1-1.732/2,3/2) -- (1-1.732,1) -- (1-1.732,0) -- (1-1.732/2,-1/2) -- (1,0);

\node at (1,0) [draw, shape=circle, fill=black, minimum size=3pt, inner sep=0pt] {};
\node [below right]  at (1,0) {\tiny $(4,2,6)$};
\node at (1,1/2) [draw, shape=circle, fill=black, minimum size=3pt, inner sep=0pt] {};
\node at (1,1) [draw, shape=circle, fill=black, minimum size=3pt, inner sep=0pt] {};
\node [above right] at (1,1) {\tiny $(2,4,6)$};
\node at (1-1.732/4,5/4) [draw, shape=circle, fill=black, minimum size=3pt, inner sep=0pt] {};
\node at (1-1.732/2,3/2) [draw, shape=circle, fill=black, minimum size=3pt, inner sep=0pt] {};
\node [above] at (1-1.732/2,3/2) {\tiny $(2,6,4)$};
\node at (1-3*1.732/4,5/4) [draw, shape=circle, fill=black, minimum size=3pt, inner sep=0pt] {};
\node at (1-1.732,1) [draw, shape=circle, fill=black, minimum size=3pt, inner sep=0pt] {};
\node [above left] at (1-1.732,1) {\tiny $(4,6,2)$};
\node at (1-1.732,1/2) [draw, shape=circle, fill=black, minimum size=3pt, inner sep=0pt] {};
\node at (1-1.732,0) [draw, shape=circle, fill=black, minimum size=3pt, inner sep=0pt] {};
\node [below left] at (1-1.732,0) {\tiny$(6,4,2)$};
\node at (1-3*1.732/4,-1/4) [draw, shape=circle, fill=black, minimum size=3pt, inner sep=0pt] {};
\node at (1-1.732/2,-1/2) [draw, shape=circle, fill=black, minimum size=3pt, inner sep=0pt] {};
\node [below] at (1-1.732/2,-1/2) {\tiny $(6,2,4)$};
\node at (1-1.732/4,-1/4) [draw, shape=circle, fill=black, minimum size=3pt, inner sep=0pt] {};

\node at (1-1.732/2,1/2) [draw, shape=circle, fill=black, minimum size=3pt, inner sep=0pt] {};

\node at (1-1.732/4,1/4) [draw, shape=circle, fill=black, minimum size=3pt, inner sep=0pt] {};
\node at (1-1.732/4,3/4) [draw, shape=circle, fill=black, minimum size=3pt, inner sep=0pt] {};
\node at (1-1.732/2,1) [draw, shape=circle, fill=black, minimum size=3pt, inner sep=0pt] {};
\node at (1-3*1.732/4,3/4) [draw, shape=circle, fill=black, minimum size=3pt, inner sep=0pt] {};
\node at (1-3*1.732/4,1/4) [draw, shape=circle, fill=black, minimum size=3pt, inner sep=0pt] {};
\node at (1-1.732/2,0) [draw, shape=circle, fill=black, minimum size=3pt, inner sep=0pt] {};

\node at (0.1,-1.2) {$P$};
\end{tikzpicture}
\hspace{20bp}
\raisebox{46bp}{$=$}
\hspace{20bp}
\begin{tikzpicture}[scale=1.2]
\draw[thick,fill=green!20, fill opacity=.5] (1,0) -- (1,1/2) -- (1-3*1.732/4,5/4) -- (1-1.732,1) -- (1-1.732,0) -- (1-1.732/2,-1/2) -- (1,0);

\node at (1,0) [draw, shape=circle, fill=black, minimum size=3pt, inner sep=0pt] {};
\node [below right]  at (1,0) {\tiny $(4,2,6)$};
\node at (1,1/2) [draw, shape=circle, fill=black, minimum size=3pt, inner sep=0pt] {};
\node [above right] at (1,1/2) {\tiny $(3,3,6)$};
\node at (1-3*1.732/4,5/4) [draw, shape=circle, fill=black, minimum size=3pt, inner sep=0pt] {};
\node [above] at (1-3*1.732/4,5/4) {\tiny $(3,6,3)$};
\node at (1-1.732,1) [draw, shape=circle, fill=black, minimum size=3pt, inner sep=0pt] {};
\node [above left] at (1-1.732,1) {\tiny $(4,6,2)$};
\node at (1-1.732,1/2) [draw, shape=circle, fill=black, minimum size=3pt, inner sep=0pt] {};
\node at (1-1.732,0) [draw, shape=circle, fill=black, minimum size=3pt, inner sep=0pt] {};
\node [below left] at (1-1.732,0) {\tiny$(6,4,2)$};
\node at (1-3*1.732/4,-1/4) [draw, shape=circle, fill=black, minimum size=3pt, inner sep=0pt] {};
\node at (1-1.732/2,-1/2) [draw, shape=circle, fill=black, minimum size=3pt, inner sep=0pt] {};
\node [below] at (1-1.732/2,-1/2) {\tiny $(6,2,4)$};
\node at (1-1.732/4,-1/4) [draw, shape=circle, fill=black, minimum size=3pt, inner sep=0pt] {};

\node at (1-1.732/2,1/2) [draw, shape=circle, fill=black, minimum size=3pt, inner sep=0pt] {};

\node at (1-1.732/4,1/4) [draw, shape=circle, fill=black, minimum size=3pt, inner sep=0pt] {};
\node at (1-1.732/4,3/4) [draw, shape=circle, fill=black, minimum size=3pt, inner sep=0pt] {};
\node at (1-1.732/2,1) [draw, shape=circle, fill=black, minimum size=3pt, inner sep=0pt] {};
\node at (1-3*1.732/4,3/4) [draw, shape=circle, fill=black, minimum size=3pt, inner sep=0pt] {};
\node at (1-3*1.732/4,1/4) [draw, shape=circle, fill=black, minimum size=3pt, inner sep=0pt] {};
\node at (1-1.732/2,0) [draw, shape=circle, fill=black, minimum size=3pt, inner sep=0pt] {};

\node at (0.15,-1.2) {$P_T$};
\end{tikzpicture}
\hspace{20bp}
\raisebox{46bp}{$+$}
\hspace{20bp}
\begin{tikzpicture}[scale=1]
\draw[thick,fill=green!20, fill opacity=.5] (1,1) -- (1-1.732/2,3/2);

\node at (1,1) [draw, shape=circle, fill=black, minimum size=3pt, inner sep=0pt] {};
\node at (1-1.732/4,5/4) [draw, shape=circle, fill=black, minimum size=3pt, inner sep=0pt] {};
\node at (1-1.732/2,3/2) [draw, shape=circle, fill=black, minimum size=3pt, inner sep=0pt] {};
\node [above right] at (1,1) {\tiny $(2,4,6)$};
\node [above] at (1-1.732/2,3/2) {\tiny $(2,6,4)$};
\node [below] at (1-1.732/2,-1/2) {\phantom{5}};

\node at (0.5,-1.2) {$P|T\times P/T$};
\end{tikzpicture}
\end{center}

\caption{Sliding a facet again.}\label{fig:slidingfacetagain}
\end{figure}
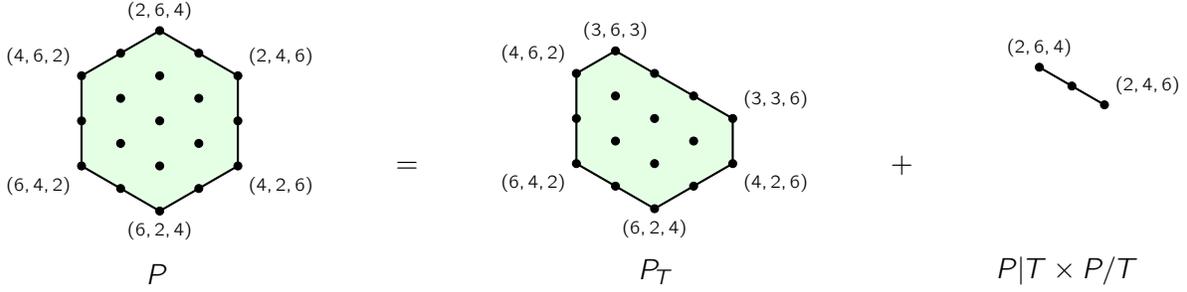

Suppose that $T$ is a nonempty, proper subset of $E$ with $|T|=k$ and define
\[
S_{E,T} \, \coloneqq  \, \left\{\sigma\in S_E:\sigma(i)\in T \text{ for all } i\leq k \right\} .
\]
If $P\subset\R^E$ is a nondegenerate generalized permutohedron, then the vertices of $P$ indexed by $S_{E,T}$ are those that lie in the facet of $P$ that maximizes the dot product with $e_T$. For example, if $E=[3]$ and $T=\{2,3\}$, then $S_{E,T}=\{231,321\}$, and if $P$ is the generalized permutohedron in Figure~\ref{fig:slidingfacetagain}, then $v_P(231)=(2,6,4)$ while $v_P(321)=(2,4,6)$, and these are precisely the vertices of $P$ that maximize the dot product with $e_T=(0,1,1)$.

For any piecewise Laurent monomial $f\in\PLM(\Sigma_E)^*$, we define $f_T\in\PLM(\Sigma_E)^*$ by
\[
(f_T)_\sigma \, \coloneqq  \, \begin{cases}
f_\sigma & \text{ if } \sigma\notin S_{E,T} \, , \\
\frac{x_{\sigma(k+1)}}{x_{\sigma(k)}}f_\sigma & \text{ if } \sigma\in S_{E,T} \, .
\end{cases}
\]
If $f=f_P$ for some nondegenerate generalized permutohedron $P\subset\R^E$, then $f_T$ is the piecewise Laurent monomial associated to the generalized permutohedron obtained by sliding in the facet indexed by $T$ by one lattice step along its normal direction $e_T$. For example, if $P$ is the generalized permutohedron in Figure~\ref{fig:slidingfacetagain}, then $f_T$ is the piecewise Laurent monomial of the generalized permutohedron $P_T$ appearing in Figure~\ref{fig:slidingfacetagain}. Notice that the vertex $v_P(231)=(2,6,4)$ of $P$ was shifted by $e_1-e_3$ to the vertex $v_{P_T}(231)=(3,6,3)$ of $P_T$, while the vertex $v_P(321)=(2,4,6)$ was shifted by $e_1-e_2$ to the vertex $v_{P_T}(321)=(3,3,6)$ of $P_T$.

For each $\sigma\in S_{E,T}$, consider the bijections $\sigma|T:\{1,\dots,k\}\rightarrow T$ and $\sigma/T:\{1,\dots,n-k\}\rightarrow T^c$ defined by $(\sigma|T)(i)=\sigma(i)$ and $(\sigma/T)(i)=\sigma(i+k)$, respectively. There is a natural bijection
\begin{align*}
S_{E,T}&\rightarrow S_T\times S_{T^c}\\
\sigma&\mapsto (\sigma|T,\sigma/T) \, .
\end{align*}
For any $f\in\PLM(\Sigma_E)^*$, define $f|T\in\PLM(\Sigma_T)$ and $f/T\in\PLM(\Sigma_{T^c})$ to be the piecewise Laurent
monomials determined by $(f|T)_{\sigma|T}=f_\sigma|_{\{x_i=1: i\notin T\}}$ and $(f/T)_{\sigma/T}=f_\sigma|_{\{x_i=1 :
i\in T\}}$ for any $\sigma\in S_{E,T}$. Notice that $f|T$ is well defined (i.e., independent of $\sigma\in S_{E,T}$) because given any $\sigma,\sigma'\in S_{E,T}$ with $\sigma|T=\sigma'|T$, there is a sequence of adjacent transpositions $\tau_{i_1},\dots,\tau_{i_\ell}$ with $i_1,\dots,i_\ell>k$ such that $\sigma=\tau_{i_1}\circ\dots\circ\tau_{i_\ell}$, but then the fact that $f\in\PLM(\Sigma_E)^*$ implies that the difference $f_\sigma-f_{\sigma'}$ vanishes upon specializing at $\{x_i=1: i\notin T\}$. Similarly, $f/T$ is well defined. By construction, for any $f\in\PLM(\Sigma_E)^*$ and $\sigma\in S_{E,T}$, we have
\begin{equation}\label{eq:monomialproduct}
f_\sigma \, = \, (f|T)_{\sigma|T} (f/T)_{\sigma/T} \, .
\end{equation}
If $f=f_P$ for a nondegenerate generalized permutohedron $P\subset\R^E$, then $f|T=f_{P|T}$ and $f/T=f_{P/T}$ for some nondegenerate generalized permuothedra $P|T\subset\R^T$ and $P/T\subset\R^{T^c}$, and moreover, for every $\sigma\in S_{E,T}$, the identity \eqref{eq:monomialproduct} reflects the fact that
\[
v_P(\sigma) \, = \, v_{P|T}(\sigma|T)\times v_{P/T}(\sigma/T) \, .
\]
For example, with $P$ as above and $T=\{2,3\}$, 
\[
P|T \, = \, \mathrm{conv}\left\{v_{P|T}(23) = (6,4), \ v_{P|T}(32) = (4,6) \right\} \, \subset \, \R^T, 
\]
\[
P/T \, = \, \mathrm{conv} \left\{v_{P/T}(1)=(2) \right\} \, \subset \, \R^{T^c},
\]
and $v_P(231)=v_{P|T}(23)\times v_{P/T}(1)=(2,6,4)$ and $v_P(321)=v_{P|T}(32)\times v_{P/T}(1)=(2,4,6)$. 

We have now introduced the precise notation required for the statement and proof of Theorem~\ref{thm:recursion}, which we restate below in its most general form.

\begin{theorem}\label{thm:recursionrestated}
Let $E$ be a set of size $n$. For any piecewise Laurent monomial $f\in\PLM(\Sigma_E)^*$, for any matroid $\M$ on $E$, and for any nonempty proper subset $\emptyset\subset T\subset E$, 
\[
Q_\M(f) \, = \, Q_{\M}(f_T)+Q_{\M|T}(f|T) \, Q_{\M/T}(f/T) \, .
\]
\end{theorem}

It follows from the discussion above that, if $\M$ is the Boolean matroid and $f=f_P$ is the piecewise Laurent monomial associated to a nondegenerate generalized permutohedron $P$, then the conclusion of Theorem~\ref{thm:recursionrestated} specializes (by Brion's Formula) to a concrete statement about lattice points:
\[
q(P) \, = \, q(P_T)+q(P|T) \, q(P/T) \, ,
\]
as depicted in Figure~\ref{fig:slidingfacetagain}.

\begin{proof}
Assume $|T|=k$. By definition, 
\begin{align*}
Q_M(f)-Q_M(f_T) \, &=\sum_{\sigma\in S_E}f_\sigma \, Q(C_\sigma)\prod_{i\notin B_\M(\sigma)}(1-x_i)-\sum_{\sigma\in S_E}(f_T)_\sigma \, Q(C_\sigma)\prod_{i\notin B_\M(\sigma)}(1-x_i)\\
&=\sum_{\sigma\in S_{E,T}}f_\sigma \Big(1-\frac{x_{\sigma(k+1)}}{x_{\sigma(k)}}\Big)Q(C_\sigma)\prod_{i\notin B_\M(\sigma)}(1-x_i) \, .
\end{align*}
Using the notation introduced prior to the statement of the theorem, for any $\sigma\in S_{E,T}$, we have
\[
\left(1-\frac{x_{\sigma(k+1)}}{x_{\sigma(k)}}\right)Q(C_\sigma) \, = \,
\frac{1}{\prod_{i=1}^{k-1}\big(1-\frac{x_{\sigma(i+1)}}{x_{\sigma(i)}}\big)}\cdot\frac{1}{\prod_{i=k+1}^{n-1}\big(1-\frac{x_{\sigma(i+1)}}{x_{\sigma(i)}}\big)}
\, = \, Q(C_{\sigma|T}) \, Q(C_{\sigma/T})
\]
and
\[
f_\sigma \, = \, (f|T)_{\sigma|T} (f/T)_{\sigma/T} \, .
\]
Furthermore, by definition of restriction and contraction, it follows that
\[
B_\M(\sigma) \, = \, B_{\M|T}(\sigma|T)\cup B_{\M/T}(\sigma/T) \, , 
\]
so that
\[
\prod_{i\in E \setminus B_\M(\sigma)}(1-x_i) \, =\prod_{i\in T\setminus B_{\M|T}(\sigma|T)}(1-x_i)\prod_{i\in
T^c\setminus B_{\M/T}(\sigma/T)}(1-x_i) \, .
\]
Pulling everything together, we conclude that
\begin{align*}
Q_M(f)-Q_M(f_T) \, &=\sum_{\sigma\in S_T}(f|T)_{\sigma} \, Q(C_{\sigma})\prod_{i\in T\setminus B_{\M|T}(\sigma)}(1-x_i) \sum_{\sigma\in S_{T^c}}(f/T)_{\sigma}Q(C_{\sigma})\prod_{i\in T^c\setminus B_{\M/T}(\sigma)}(1-x_i)\\
&= \, Q_{\M|T}(f|T) \, Q_{\M/T}(f/T) \, .\qedhere
\end{align*}
\end{proof}

In the remainder of this section, we utilize Theorem~\ref{thm:recursionrestated} to connect the Laurent polynomials $Q_\M(f)$ to matroid Euler characteristics, introduced by Larson, Li, Payne, and Proudfoot in \cite{larsonetal}. We begin by recalling general definitions and facts about $K$-rings and Euler characteristics of matroids.

To develop what we require, we now shift from thinking about matroids in terms of bases to thinking about them in terms of flats. Given a matroid $\M$ on $E$, a subset $F\subseteq E$ is called a \emph{flat} if $\rk_\M(F)<\rk_\M(F\cup\{i\})$ for any $i\in E\setminus F$. Consider the unimodular fan $\Sigma_\M$ in $\R^E$ comprised of cones of the form
\[
\R_{\geq 0} \left( e_{F_1},\dots, e_{F_k} \right)
\]
where $\emptyset\subset F_1\subset \dots\subset F_k\subseteq E$ are nonempty flats of $\M$. The quotient of $\Sigma_\M$ by $\R1$ is the \emph{Bergman fan of $\M$}, but for our purposes it will be more useful to work directly with the unquotiented fan $\Sigma_\M$ in $\R^E$. Let $|\Sigma_\M|$ denote the support of $\Sigma_\M$, and notice that the fan structure on $\Sigma_\M$ is inherited from the braid fan: $\Sigma_\M=|\Sigma_\M|\cap\Sigma_E$.

We say that a function $f:|\Sigma_\M|\rightarrow\R$ is \emph{integral linear on $\Sigma_\M$} if $f$ is the restriction of some integral linear function on $\R^E$, we say that $f$ is \emph{(integral) piecewise linear on $\Sigma_\M$} if the restriction of $f$ to each cone of $\Sigma_\M$ is the restriction of some integral linear function on $\R^E$, and we say that $f$ is \emph{(integral) piecewise exponential on $\Sigma_\M$} if, for each cone $C\in\Sigma_\M$, 
\[
f|_C \, = \, \left.\left( a_1e^{m_1}+\dots+a_ke^{m_k} \right)\right|_C
\]
for some integers $a_1,\dots,a_k\in\Z$ and some integral linear functions $m_1,\dots,m_k:\R^E\rightarrow\R$. Let $\L(\Sigma_\M)$, $\PL(\Sigma_\M)$, and $\PE(\Sigma_\M)$ denote the group of integral linear functions, the group of integral piecewise linear functions, and the ring of integral piecewise exponential function on $\Sigma_\M$, respectively, and define the quotients 
\[
\uPL(\Sigma_\M))=\frac{\PL(\Sigma_\M)}{\L(\Sigma_\M)}\;\;\;\text{ and
}\;\;\;\uPE(\Sigma_\M)=\frac{\PE(\Sigma_\M)}{\langle e^f-1:f\in \L(\Sigma_M)\rangle} \, .
\]
We view $\PL(\Sigma_\M)$ and $\uPL(\Sigma_\M)$ as multiplicative subgroups of $\PE(\Sigma_\M)$ and $\uPE(\Sigma_\M)$ via the inclusion $f\mapsto e^f$, and we note that an argument similar to the proof of Lemma~\ref{lemma:span} implies that the ring of piecewise exponential functions is spanned by the group of (exponentials of) piecewise linear functions~\cite{chan2025kringssmoothtoricvarieties}. For each nonempty flat $F$, we let $\delta_F\in\PL(\Sigma_\M)$ denote the unique piecewise linear function on $\Sigma_\M$ that takes value $1$ at $e_F$ and vanishes on all other rays of $\Sigma_\M$. The functions $\delta_F$ (freely) generate $\PL(\Sigma_\M)$ as a group and their exponentials $e^{\delta_F}$ generate $\PE(\Sigma_\M)$ as a ring.

It is known that the Grothendieck $K$-ring of vector bundles on the smooth toric variety $X_{\Sigma_\M}$ is naturally isomorphic to $\uPE(\Sigma_\M)$ \cite{BV97,Mer97}. For the purposes of this paper, we use this as motivation to simply define the \emph{$K$-ring of a matroid $\M$} as the quotient ring of piecewise exponential function on~$\Sigma_\M$:
\[
K(\M) \, \coloneqq  \, \uPE(\Sigma_\M) \, .
\] 

\begin{remark}
In \cite{larsonetal}, the $K$-ring of a loopless matroid was defined as the $K$-ring of the toric variety associated to the Bergman fan $\Sigma_\M/\R 1$. Upon observing that the toric variety $X_{\Sigma_\M}$ is the total space of a line bundle over the toric variety $X_{\Sigma_\M/\R 1}$, it follows that the two varieties have canonically isomorphic $K$-rings: $K(X_{\Sigma_\M})=K(X_{\Sigma_\M/\R1})$. Thus, the definition of $K(\M)$ in \cite{larsonetal} agrees with the one given here.
\end{remark}

If $X$ is a smooth complete variety, taking holomorphic Euler characteristics of vector bundles gives rise to an additive map $\chi:K(X)\rightarrow\Z$, whose value on any vector bundle is the alternating sum of dimensions of the sheaf cohomology groups. If $X$ is not complete, then this description no longer makes sense, as the sheaf cohomology groups generally have infinite dimensions, so we should not expect a nontrivial additive map $K(X)\rightarrow\Z$ to exist. In particular, since $|\Sigma_\M|\neq\R^E$, the toric variety $X_{\Sigma_\M}$ is not complete, and we should not expect a nontrivial additive map $K(\M)\rightarrow\Z$. Nonetheless, Larson, Li, Payne, and Proudfoot \cite{larsonetal} produce a nontrivial additive map
\[
\chi_\M:K(M)\rightarrow\Z
\]
for any loopless matroid $\M$; we refer to the map $\chi_\M$ as the \emph{Euler characteristic of $\M$}. 

We now give an explicit recursive description of $\chi_\M$ that follows from the results in \cite{larsonetal}. Consider the group homomorphism
\begin{align*}
\psi_\M:\PL(\Sigma_E)&\rightarrow K(\M)\\
f&\mapsto \big[e^f|_{|\Sigma_\M|}\big].
\end{align*}
Recall that $\PL(\Sigma_E)$ is freely generated by $\{\delta_T:\emptyset\subset T\subseteq E\}$. For each 
\[
f \, =\sum_{\emptyset\subset T\subseteq E}a_T\delta_T \, \in \, \PL(\Sigma_E)
\]
and each subset $F\subseteq E$, define 
\[
f|F \, =\sum_{\emptyset\subset T\subseteq F}a_T\delta_T\in\PL(\Sigma_F)\;\;\;\text{ and }\;\;\;f/F \,
=\sum_{\emptyset\subset T\subseteq F^c} a_{F\cup T}\delta_T \, \in \, \PL(\Sigma_{F^c}) \, .
\]
Then the collection of matroid Euler characteristics $\{\chi_\M:K(\M)\rightarrow\Z:\M\text{ a loopless matroid}\}$ can be characterized as the unique collection of additive functions for which each composition $\chi_\M^*\coloneqq \chi_\M\circ\psi_\M$ satisfies the following two properties (see \cite[Proposition~8.6]{larsonetal}):
\begin{enumerate}
\item $\chi_\M^*(0)=1$, and
\item for any nonempty proper flat $F$ of $\M$,
\[
\chi_\M^*(f) \, = \, \chi_\M^*(f-\delta_F)+\chi_{\M|F}^*(f|F) \, \chi_{\M/F}^*(f/F) \, .
\]
\end{enumerate}
That there can be at most one collection of additive functions satisfying these two properties follows from induction on the size of the ground set, the fact that $\mathrm{im}(\psi_\M)$ spans $K(\M)$, and the observation that every element of $\PL(\Sigma_E)$ can be obtained from $0$ by successively adding and subtracting functions of the form $\delta_T$. The existence of a collection of additive functions satisfying these two properties, on the other hand, is far from obvious. The proof of existence developed in \cite{larsonetal} begins by restricting to the class of representable matroids, where $\chi_\M$ can be defined as the holomorphic Euler characteristics of vector bundles on a corresponding wonderful compactification, and therefore satisfies conditions (1) and (2) above, and then extending to all matroids using valuativity. Theorem~\ref{thm:euler} below gives a direct proof of existence that avoids algebraic geometry and valuativity arguments.

In order use Theorem~\ref{thm:recursionrestated} to study matroid Euler characteristics, we need to translate between piecewise Laurent polynomials and piecewise exponential functions. To do so, we note that there is a natural surjective ring homomorphism
\[
\phi_\M:\PLP(\Sigma_E)\rightarrow K(\M)
\]
induced by sending each monomial $x^m$ to the exponential function $e^m$, then restricting the domain to $|\Sigma_\M|$, and finally taking the coset by the ideal in the quotient description of $K(\M)=\uPE(\Sigma_\M)$. With notation now established, we state and prove a direct construction of matroid Euler characteristics.

\begin{theorem}\label{thm:euler}
For every matroid $\M$ on $E$, there exists a well-defined additive map $\widetilde \chi_\M:K(\M)\rightarrow \Z$ determined by
\[
\widetilde \chi_\M(\phi_\M(f)) \, = \, Q_\M(f)|_{x=1} \, .
\]
If $\M$ has loops, then $\widetilde\chi_\M=0$. If $\M$ is loopless, then the composition $\widetilde\chi_\M^*\coloneqq \widetilde \chi_\M\circ\psi_\M$ satisfies
\begin{enumerate}
\item $\widetilde\chi_\M^*(0)=1$, and
\item for any nonempty proper flat $F$ of $\M$,
\[
\widetilde\chi_\M^*(f) \, = \,
\widetilde\chi_\M^*(f-\delta_F)+\widetilde\chi_{\M|F}^*(f|F) \, \widetilde\chi_{\M/F}^*(f/F) \, .
\]
\end{enumerate}
In particular, if $\M$ is loopless, then $\widetilde\chi_\M=\chi_\M$, where the latter is the matroid Euler characteristic introduced by Larson, Li, Payne, and Proudfoot \cite{larsonetal}.
\end{theorem}

\begin{proof}
First, observe that if $\M$ has a loop $i\in E$, then the definition of $Q_\M(f)$ implies that $1-x_i$ divides $Q_\M(f)$ for every $f\in\PLP(\Sigma_E)$, so $Q_\M(f)|_{x=1}=0$. Thus, when $\M$ has loops, the content of the theorem is trivial: the additive map $\widetilde\chi_\M$ is just the zero function. Henceforth, we assume that $\M$ is loopless.

To argue that the additive map $\widetilde \chi_\M:K(\M)\rightarrow \Z$ exists, we begin with the additive map 
\begin{align*}
\widehat\chi_\M:\PLP(\Sigma_E)&\rightarrow \Z\\
f&\mapsto Q_\M(f)|_{x=1} \, .
\end{align*}
To show that $\widehat\chi_\M$ induces an additive map on $K(\M)$, we must prove that $\ker(\phi_\M)\subseteq\ker(\widehat\chi_\M)$. We can view $\phi_\M$ as a composition of the following three ring homomorphisms 
\[
\PLP(\Sigma_E)\rightarrow\PE(\Sigma_E)\rightarrow\PE(\Sigma_\M)\rightarrow\uPE(\Sigma_\M)=K(\M) \, ,
\]
where the first is an isomorphism sending each monomial $x^m$ to the corresponding exponential function $e^m$, the second map restricts the domain, and the third map quotients by $\langle e^f-1:f\in \L(\Sigma_\M)\rangle$. Recalling that $\PLP(\Sigma_E)$ is spanned by piecewise Laurent monomials $x^f$ with $f\in\PL(\Sigma_E)$, it follows from the above compositional description of $\phi_\M$ that $\ker(\phi_\M)$ is spanned by
\[
\left\{x^{f+m}-x^f:f\in\PL(\Sigma_E),\;m\in\L(\Sigma_E)\right\}\cup\left\{x^f-x^g:f,g\in\PL(\Sigma_E),\;f|_{|\Sigma_\M|}=g|_{|\Sigma_\M|}\right\}.
\]
Since $Q_\M(x^{f+m})=x^mQ_\M(x^f)$ for any $m\in\L(\Sigma_E)$, it follows that the first set is contained within $\ker(\widehat\chi_\M)$. Now consider an element $x^f-x^g$ from the second set. Since $f$ and $g$ agree once restricted to $|\Sigma_\M|$, it follows that
\[
f-g \, \in \, \Z\left\{\delta_T:\emptyset\subset T\subseteq E,\;T\text{ not a flat of }\M\right\}.
\]
Thus, it suffices to prove that $\widehat\chi_\M(x^f)=\widehat\chi_\M(x^{f-\delta_T})$ for any nonempty subset $T\subseteq E$ that is not a flat of $\M$. Suppose that $|T|=k$. Since shifting $f$ by $\delta_T$ only changes the values of the function on cones containing $e_T$, which are those cones indexed by $S_{E,T}$, a direct computation shows that
\[
(x^{f-\delta_T})_\sigma \, = \, \begin{cases}
(x^f)_\sigma & \text{ if } \sigma\notin S_{E,T} \, , \\
\frac{x_{\sigma(k+1)}}{x_{\sigma(k)}}(x^f)_\sigma & \text{ if } \sigma\in S_{E,T} \, .
\end{cases}
\]
In other words, in the notation of Theorem~\ref{thm:recursionrestated}, we have $x^{f-\delta_T}=(x^f)_T$, so 
\[
Q_\M(x^f) \, = \, Q_\M(x^{f-\delta_T})+Q_{\M|T}(x^f|T) \, Q_{\M/T}(x^f/T) \, .
\]
As $T$ is not a flat of $\M$, it follows that $\M/T$ is a matroid with loops, so $Q_{\M/T}(x^f/T)|_{x=1}=0$, implying that  $Q_\M(x^f)|_{x=1}=Q_\M(x^{f-\delta_T})|_{x=1}$. Thus, $\widehat\chi_\M(x^f)=\widehat\chi_\M(x^{f-\delta_T})$, concluding the proof that $\ker(\phi_\M)$ is contained in $\ker(\widehat\chi_\M)$. In summary, we have now argued the existence of the additive map $\widetilde\chi_\M:K(\M)\rightarrow\Z$.

It remains to verify Conditions (1) and (2) in the statement of the theorem. We first compute $\widetilde\chi_\M^*(0)=Q_\M(1)_{x=1}=1$, where the second equality follows from Corollary~\ref{corollary:basecase}, verifying Condition (1). Condition (2) is a direct consequence of Theorem~\ref{thm:recursionrestated}, using once again that $x^{f-\delta_F}=(x^f)_F$.
\end{proof}

\begin{remark}
In recent work of Eur, Fink, and Larson \cite{EurFinkLarson}, it is shown that $\chi^*_\M(f_P)>0$ for every loopless
matroid $\M$ on $E$ and every generalized permutohedron $P\subset\R^E$. More generally, letting $d$ denote the degree
of the polynomial $k\mapsto\chi^*_\M(f_P^k)$ and defining the quantities $h_0^*,\dots,h_d^*$ by the equation
\[
\sum_{k\geq 0} \chi_\M^*(f_P^k)t^k \, = \, \frac{h_0^*+h_1^*t+\dots+h_d^*t^d}{(1-t)^{d+1}} \, ,
\]
Eur, Fink, and Larson prove that $(h_0^*,\dots,h_d^*)$ is a Macaulay vector \cite[Theorem~B]{EurFinkLarson}, which
implies that $h_0^*,\dots,h_d^*\geq 0$, and this, along with the fact that $h_0^*=1$, further implies that $\chi_\M^*(f_P)>0$. We think it would be very interesting if one could combinatorially derive these positivity result from the interpretation of $\chi_\M^*(f_P)$ as the specialization of the Laurent polynomial $Q_\M(f_P)$.
\end{remark}

\begin{remark}
In \cite{ehrhartfans}, the authors introduce the notion of an \emph{Ehrhart fan}, which essentially amounts to the existence of an additive map from the $K$-ring of the associated toric variety to the integers satisfying the natural generalization of the two conditions (1) and (2) discussed above. In \cite[Theorem~6.1]{ehrhartfans}, it is proved that Bergman fans of loopless matroids are Ehrhart, which essentially follows from \cite[Proposition~8.6]{larsonetal}. Theorem~\ref{thm:euler} gives a combinatorial justification of the Ehrharticity of Bergman fans, avoiding the algebraic geometry and valuativity arguments that were used in \cite{larsonetal}.
\end{remark}


\section{Reciprocity}\label{sec:reciprocity}

In this section, we prove Theorem~\ref{thm:reciprocity}, relate it to Ehrhart reciprocity for generalized permutohedra, and use it to give a combinatorial proof of Serre duality for matroid Euler characteristics. We begin by restating Theorem~\ref{thm:reciprocity} for convenience. Recall that $(-)^\vee:\PLP(\Sigma_E)\rightarrow \PLP(\Sigma_E)$ is the ring homomorphism that inverts all variables: $x_i\mapsto x_i^{-1}$.

\begin{theorem}\label{thm:reciprocityrecalled}
Let $E$ be a set of size $n$ and let $\M$ be a matroid on $E$. Define $\omega_\M\in\PLP(\Sigma_E)$ by
\[
(\omega_\M)_\sigma \, \coloneqq  \, \frac{x_{\sigma(n)}}{x_{\sigma(1)}}\prod_{i\notin B_\M(\sigma)}x_i^{-1} \, .
\]
Then for any piecewise Laurent polynomial $f\in\PLP(\Sigma_E)$, 
\[
Q_\M(f^\vee) \, = \, (-1)^{\rk(\M)-1} \, Q_\M(f\cdot\omega_\M)^\vee.
\]
\end{theorem}

\begin{proof}
This follows from a direct computation using the definition of $Q_\M$:
\begin{align*}
Q_\M(f^\vee) \, &= \sum_{\sigma\in S_E}\frac{f_\sigma^\vee}{\big(1-\frac{x_{\sigma(2)}}{x_{\sigma(1)}}\big)\cdots\big(1-\frac{x_{\sigma(n)}}{x_{\sigma(n-1)}}\big)}\prod_{i\notin B_\M(\sigma)}(1-x_i)\\
&= \, (-1)^{\rk(\M)-1}\sum_{\sigma\in S_E}\frac{f_\sigma^\vee\frac{x_{\sigma(1)}}{x_{\sigma(n)}}\prod_{i\notin B_\M(\sigma)}x_i}{\big(1-\frac{x_{\sigma(1)}}{x_{\sigma(2)}}\big)\cdots\big(1-\frac{x_{\sigma(n-1)}}{x_{\sigma(n)}}\big)}\prod_{i\notin B_\M(\sigma)}(1-x_i^{-1})\\
&= \, (-1)^{\rk(\M)-1} \, Q_\M(f\cdot\omega_\M)^\vee.\qedhere
\end{align*}
\end{proof}

We now describe how Theorem~\ref{thm:reciprocityrecalled} generalizes Ehrhart reciprocity for generalized permutohedra.
Recall that the \emph{Ehrhart function} of a lattice polytope $P\subset\R^n$ is the function whose value at a positive
integer $k$ is the the number of lattice points in the $k$th dilate of $P$: $k\mapsto |kP\cap\Z^n|$. It is a
fundamental result of Ehrhart \cite{ehrhartpolynomial} that the Ehrhart function is the specialization of a polynomial
$E_P(t)\in\Q[t]$, called the \emph{Ehrhart polynomial} of $P$. A cornerstone of Ehrhart theory is \emph{Ehrhart
reciprocity}, which gives an interpretation of the value of the Ehrhart polynomial at $-1$ in terms of the relative
interior $P^\circ$ of $P$ \cite{ehrhartpolynomial,macdonald} (see also \cite{crt}):
\[
E_P(-1) \, = \, (-1)^{\dim(P)} \left|P^\circ\cap\Z^n\right|.
\]
More generally, Ehrhart reciprocity is the statement that $E_P(-k) \, = \, (-1)^{\dim(P)} \left|(kP)^\circ\cap\Z^n\right|$, but the $k>1$ case follows from the $k=1$ case upon replacing $P$ with $kP$.

In the case of a generalized permutohedron $P\subset\R^n$, Brion's Formula allows us to interpret the Ehrhart polynomial as $E_P(k)=Q(f_P^k)|_{x=1}$ where
\[
Q(f) \, \coloneqq  \sum_{\sigma\in S_E}f_\sigma \, Q(C_\sigma)\, .
\]
In particular, $E_P(-1)=Q(f_P^\vee)|_{x=1}$. Moreover, if we assume that $P$ has  dimension $n-1$, then
\[
P=\bigcap_{\sigma\in S_n}(v_\sigma(P)+C_\sigma) \; \Longrightarrow \; P^{\circ}=\bigcap_{\sigma\in S_n}(v_\sigma(P)+C_\sigma)^\circ,
\]
and since $C_\sigma$ is unimodular, the lattice points in $(v_\sigma(P)+C_\sigma)^\circ$ are the same as the lattice points in the translation of $v_\sigma(P)+C_\sigma$ by the sum of ray generators $e_{\sigma(2)}-e_{\sigma(1)},\dots,e_{\sigma(n)}-e_{\sigma(n-1)}$. It follows that
\[
P^\circ\cap\Z^n \, = \, \left(\bigcap_{\sigma\in S_n}
\big(v_\sigma(P)+e_{\sigma(n)}-e_{\sigma(1)}+C_\sigma\big)\right)\cap \Z^n \, = \, \left. Q(f_P\cdot\omega)
\right|_{x=1} \, ,
\]
where $\omega_\sigma=\frac{x_{\sigma(n)}}{x_{\sigma(1)}}$. Thus, Ehrhart reciprocity in this setting translates to the formula
\[
Q(f_P^\vee)|_{x=1}=(-1)^{n-1}Q(f_P\cdot\omega)|_{x=1},
\]
which is exactly the statement of Theorem~\ref{thm:reciprocityrecalled} upon specializing $\M$ to the Boolean matroid and $x=1$.

Finally, we note that Theorem~\ref{thm:reciprocityrecalled} along with Theorem~\ref{thm:euler} have the following immediate consequence in terms of Euler characteristics of matroids.

\begin{corollary}
Let $M$ be a matroid on a finite set $E$. Then for any $f\in\PLP(\Sigma_E)$, 
\[
\chi_\M(\sigma_\M(f^\vee)) \, = \, (-1)^{\rk(\M)-1}\chi_\M \left( \sigma_\M(f)\cdot\sigma_\M(\omega_\M) \right) ,
\]
where $\phi_\M:\PLP(\Sigma_E)\rightarrow K(\M)$ is the natural surjective ring homomorphism from piecewise Laurent polynomials on $\Sigma_E$ to piecewise exponential functions on $\Sigma_\M$.
\end{corollary}

This recovers Serre duality for matroid Euler characteristics (\cite[Theorem~6.2]{larsonetal}) via purely combinatorial methods, answering a question of Larson, Li, Payne, and Proudfoot \cite[Question~6.5]{larsonetal}.

\bibliographystyle{amsplain}
\bibliography{matroidbrion}

\end{document}